\documentclass[11pt]{article}
\usepackage{amsmath,amsfonts}
\usepackage{txfonts,wasysym,lscape,amssymb,mathrsfs,extarrows}
\usepackage[all,pdf]{xy}
\usepackage[colorlinks,linkcolor=black,anchorcolor=blue,citecolor=green]{hyperref}
\usepackage[titletoc,title]{appendix}
\usepackage[algo2e,vlined,ruled]{algorithm2e}

\usepackage{algorithm}
\usepackage{algorithmic}

\usepackage{fancyhdr}

\usepackage{amsfonts}
\usepackage{setspace,url}
\usepackage{graphics,graphicx,epstopdf}
\usepackage{float}
\usepackage{placeins}
\usepackage{tabularx}
\usepackage{longtable}
\usepackage{booktabs,multirow,array,multicol}
\usepackage{enumitem}
\usepackage{subfig}
\ifpdf
  \DeclareGraphicsExtensions{.eps,.pdf,.png,.jpg}
\else
  \DeclareGraphicsExtensions{.eps}
\fi
\newcommand{\expplot}[1]{\includegraphics[width=0.45\textwidth,trim=12bp 218bp 12bp 218bp,clip]{#1}}
\newtheorem{theorem}{Theorem}[section]
\newtheorem{proposition}{Proposition}[section]
\newtheorem{lemma}{Lemma}[section]
\newtheorem{corollary}{Corollary}[section]

\newtheorem{remark}{Remark}[section]

\newcommand{\R}{{\mathbb R}}

\newcommand{\be}{\begin{equation}}
\newcommand{\ee}{\end{equation}}

\title{ A Quadratic-Approximation-Based Stochastic Approximation Method for Weakly Convex Stochastic Programming\thanks{Supported by National Key R\&D Program of China under project number 2022YFA1004000 and National Natural Science Foundation of China (No. 12371298).}}

\author{
Yule Zhang\footnote{
School of   Science, Dalian Maritime University,
Dalian 116085, China.(ylzhang@dlmu.edu.cn)}, \,\,
Benqi Liu\footnote{Beijing International Center for Mathematical Research, Peking University, Beijing 100871, China.(bqliu@pku.edu.cn)},\,\,
Xiantao Xiao\footnote{Institute of Operations Research and Control Theory, School of Mathematical Sciences, Dalian University
of Technology, Dalian 116024, China.(xtxiao@dlut.edu.cn)}\,\, and Liwei Zhang\footnote{National Frontiers Science Center for Industrial Intelligence and Systems Optimization, Northeastern University, Shenyang 110819, China;  Key Laboratory of Data Analytics and Optimization for Smart Industry  (Northeastern University),  Ministry of Education, Shenyang 110819, China. (zhanglw@mail.neu.edu.cn)}}

\date{}

\begin{document}

\maketitle
 \vspace{2mm}
\begin{center}
\parbox{13.5cm}{\small
\textbf{Abstract.}
We propose a novel stochastic approximation algorithm, termed PMQSopt, for solving weakly convex stochastic optimization problems involving expectation-valued functions. The algorithm is constructed by integrating the proximal method of multipliers with quadratic approximations of the original stochastic problem. We analyze the sample complexity of PMQSopt in terms of the total number of stochastic gradient evaluations required. The convergence of the algorithm is characterized by three metrics associated with the $\epsilon$-KKT conditions: the average squared norm of the gradient of the Moreau envelope of the Lagrangian, the average constraint violation, and the average complementarity violation. For each of these metrics, we establish an expected convergence rate of $\mathcal{O}(T^{-1/4})$ after $T$ iterations. Furthermore, we show that with probability at least $1-1/T^{2/3}$, the gradient of the Lagrangian satisfies an $\mathcal{O}(T^{-1/8})$ bound; with probability at least $1-2/T^{2/3}$, the constraint violation achieves an $\mathcal{O}(T^{-1/4})$ bound; and with probability at least $1-3/T^{2/3}$, the complementarity violation attains an $\mathcal{O}(T^{-1/4})$ bound. All results are established under two mild conditions: (i) weak convexity of all problem functions, and (ii) the existence of a strictly feasible point. The proposed PMQSopt algorithm is a sequentially strongly convex programming method that is readily implementable. Numerical experiments illustrate its practical performance.
\\[10pt]
 }
\end{center}

\section{Introduction}\label{Sec1}
\setcounter{equation}{0}
\quad \, In this paper, we consider the following stochastic optimization problem
\begin{equation}\label{eq:1}
\begin{array}{rl}
\displaystyle \min_{x \in X_0} & f(x)=\mathbb{E}[F(x,\xi)]\\[4pt]
{\rm s.t.} & g_i(x)=\mathbb{E}[G_i(x,\xi)] \leq 0, i=1,\ldots,p.\\
\end{array}
\end{equation}
Here $X_0 \subset \mathbb R^n$ is a nonempty  closed convex set, $\xi$ is a random vector whose probability distribution $P$  is supported on set $\Xi \subseteq \mathbb R^q$ and $F: {\cal O}_0 \times \Xi \rightarrow \mathbb R$, $G_i:{\cal O}_0 \times \Xi \rightarrow \mathbb R$, $i=1,\ldots, p$, where ${\cal O}_0 \subset \mathbb R^n$ is an open convex set containing $X_0$. The Lagrangian of Problem (\ref{eq:1}) is defined by
\begin{equation}\label{eq:lag}
L(x,\lambda)=f(x)+\displaystyle \sum_{j=1}^p \lambda_j g_j(x), \quad (x, \lambda) \in {\cal O}_0 \times \mathbb R^p.
\end{equation}
It is well known that a fundamental computational challenge in solving the stochastic optimization problem (\ref{eq:1}) lies in the fact that the expectation is a multidimensional integral, which cannot be evaluated with high accuracy when the dimension $q$ is large. Stochastic approximation (SA) methods constitute an important class of numerical techniques for addressing this difficulty and have been intensively studied in recent years. The SA method dates back to the pioneering work of Robbins and Monro \cite{Robbins1951} on finding a root of an equation defined by a conditional expectation. It was later extended to systems of equations defined by conditional expectations by Chung \cite{Chung1954} and Sacks \cite{Sacks1958}. Since finding a stationary point of an unconstrained optimization problem is equivalent to solving a system of equations, the SA method can be readily extended to solve unconstrained stochastic optimization problems. Due to their low computational cost per iteration, SA algorithms have become widely used in stochastic optimization; see, e.g., \cite{Polyak1992, Nemirovski2009, Xiao2010,Shalev-Shwartz2011}.

An important improvement of the SA method was developed by Polyak and Juditsky \cite{Polyak1992}, who adopted an averaging technique to accelerate their SA algorithm. Subsequently, Nemirovski, Juditsky, Lan, and Shapiro \cite{Nemirovski2009} introduced a robust averaging approach and showed that, without assuming smoothness and strong convexity of the objective function, the convergence rate is ${\rm O}(t^{-1/2})$. Following this seminal work, numerous important improvements have emerged, even for nonconvex stochastic optimization problems; see, e.g., \cite{Lan2012, Lan2012a, Ghadimi2012,Ghadimi2013, Ghadimi2013a, Ghadimi2016}. In all of the aforementioned works, the feasible region is assumed to be an abstract closed convex set. Consequently, none of these SA algorithms are directly applicable to problems involving expectation constraints.

Over the past decade, numerous works on stochastic approximation methods for constrained stochastic optimization problems have been published. For convex stochastic optimization problems with expectation constraints, Yu et al. \cite{Yu2017} propose an algorithm that achieves ${\rm O}(1/\sqrt{T})$ expected regret and constraint violations, along with ${\rm O}(\log T/\sqrt{T})$ high-probability regret and constraint violations. Lan and Zhou \cite{Lan2020a} develop a cooperative SA approach that attains the optimal ${\rm O}(1/\sqrt{T})$ convergence rate for both the optimality gap and constraint violation. Xu \cite{Xu2020} introduces a primal-dual stochastic gradient method for convex programs with multiple functional constraints, achieving the optimal ${\rm O}(1/\sqrt{T})$ convergence rate for the convex case and a nearly optimal ${\rm O}(\log T / T)$ rate for the strongly convex case. Zhang et al. \cite{Zhang2023} construct a primal method of multipliers for convex stochastic optimization problems, which achieves ${\rm O}(1/\sqrt{T})$ expected regret and constraint violations, as well as ${\rm O}(\log T / \sqrt{T})$ high-probability regret and constraint violations.

For nonconvex constrained stochastic optimization problems, several interesting works have been published under different problem settings. Jin and Wang \cite{Jin2022} establish iteration and sample complexities for a primal-dual algorithm that solves a constrained optimization problem whose objective function is an expectation plus a convex function, subject to inequality constraints. Boob et al. \cite{Boob2023} develop a double-loop algorithm for a class of nonconvex constrained stochastic optimization problems and analyze its complexity in finding an $(\epsilon, \delta)$-KKT point. Li et al. \cite{Li2024} consider an expectation equality-constrained optimization problem, design inexact augmented Lagrangian methods, and provide an oracle complexity result. Curtis \cite{Curtis2024} proves a worst-case complexity bound for an SQP algorithm solving optimization problems involving a stochastic objective function and deterministic nonlinear equality constraints. Shi et al. \cite{Shi2025} consider a constrained optimization problem where the objective function is an expectation plus a convex function, with both equality and inequality constraints. They propose a momentum-based linearized augmented Lagrangian method for this problem, analyze its asymptotic properties, and establish sample complexities.

The aim of paper is to investigate the possibility of constructing a new stochastic approximation algorithm for solving the nonconvex stochastic problem (\ref{eq:1}) with good sample complexities under some reasonable assumptions about $F$ and $G_i$ for $i=1,\ldots, p$.  We arrive at this goal by constructing an algorithm  basing on the proximal method of multipliers for  quadratic approximations to  the stochastic optimization problem, and obtain good sample complexities and high prability guarantees for Karush-Kuhn-Tucker conditions under the assumption that $F$ and $G_i$ are weakly convex. For weakly convex stochastic optimization problems without functional constraints, Davis and Drusvyatskiy (2019) \cite{Davis2019} showed that the proximal stochastic subgradient method achieves a sample complexity of $\mathcal{O}(T^{-1/2})$ in terms of the squared norm of the gradient of the Moreau envelope of the objective function. However, to the best of our knowledge, no analogous results have been established for constrained problems of the form (\ref{eq:1}).

  For i.i.d. sample $\xi_1,\xi_2,\ldots$ of realizations of random vector $\xi$. Like a stochastic approximation approach for convex stochastic optimization in literatures, we consider the following  optimization problem
\begin{equation}\label{eq:Pt}
\begin{array}{ll}
\displaystyle \min_{x \in X_0} & F(x,\xi_t)\\[4pt]
{\rm s.t.} & G_i(x,\xi_t) \leq 0,i=1,\ldots,p.\\
\end{array}
\end{equation}
The quadratically constrained quadratic programming (QCQP) approximation for Problem (\ref{eq:Pt})  at $x^t$ is defined as
\begin{equation}\label{eq:LPt}
\begin{array}{ll}
\min & q^t_0(x) \\[4pt]
{\rm s.t.} &  q^t_i(x)  \leq 0,i=1,\ldots,p,\\[4pt]
 & x\in X_0,
\end{array}
\end{equation}
where $q^t_j(x),j=0,1,\ldots,p$ are quadratic approximations of $F(x,\xi_t)$ and $G_j(x,\xi_t)$, $j=1,\ldots,p$ at $x^t$, respectively. Functions $q^t_j(x),j=0,1,\ldots,p$  are defined by
\begin{equation}\label{eq:qs}
\begin{array}{l}
q^t_0(x)= F(x^t,\xi_t)+\langle \nabla_xF(x^t,\xi_t),x-x^t \rangle+\displaystyle \frac{1}{2}\langle \Sigma^t_0(x-x^t),x-x^t\rangle\\[4pt]
q^t_i(x)= G_i(x^t,\xi_t)+ \langle \nabla_xG_i(x^t,\xi_t),x-x^t \rangle +\displaystyle \frac{1}{2}\langle \Sigma^t_i(x-x^t),x-x^t\rangle,\,\,
i=1,\ldots,p.
\end{array}
\end{equation}
The augmented Lagrange function for Problem (\ref{eq:LPt})
is defined by
\begin{equation}\label{augL}
{\cal L}^t_{\sigma}(x,\lambda): =q_0^t(x) +\displaystyle \frac{1}{2\sigma}\left[ \sum_{i=1}^p[\lambda_i+\sigma q^t_i(x)]_+^2-\|\lambda\|^2\right]
\end{equation}
for $(x,\lambda)\in \mathbb R^n\times \mathbb R^p$.

 The proximal method of multipliers for Problem (\ref{eq:1}) with quadratic approximations can be described as follows, where the Hessian matrix $\Sigma^t_0$ is chosen in a smart way to be positive definite.\mbox{}\\[4pt]
{\bf PMQSopt}: A  proximal method of multipliers  with quadratic  approximations
\begin{description}
\item[Step 0 ] Input $\lambda^1=0 \in \mathbb R^p$, $x^1 \in \mathbb R^n$, a positive  integer $N$ and a set of samples $\{\xi_1,\ldots,\xi_N\}$. Positive parameters $\sigma$, $\alpha$ and $\tau$. Set $t:=1$.
\item[ Step 1] Select symmetric matrices $\Sigma^t_i$ for $i=1,\ldots, p$ and define
$$
\Sigma^t_0=-\displaystyle \sum_{i=1}^p \lambda^t_i \Sigma^t_i+\tau I.$$
\item[ Step 2] Set
 \begin{equation}\label{xna}
\begin{array}{l}
x^{t+1}= \displaystyle\hbox{arg}\min \,\left\{ {\cal L}^t_{\sigma }(x,\lambda^t) +\displaystyle\frac{\alpha}{2}\|x-x^t\|^2,x \in X_0\right\},\\[5mm]
\lambda_i^{t+1}=[\lambda_i^t+\sigma q^t_i(x^{t+1})]_+,\,\,i=1,\ldots,p.
\end{array}
\end{equation}
\item[ Step 3] Set $t:=t+1$ if $t<T$ and go to Step 1.
\item[ Output] $(x^{R},\lambda^{R})$, where $R\in [T]$ is uniformly chosen at random.
\end{description}
In the above algorithm, $[y]_+=\Pi_{\mathbb R^p_+}[y]$ denotes  the projection of $y$ on to $\mathbb R^p_+$ for any $y \in \mathbb R^p$.
 Note that the iterations $x^t=x^t(\xi_{[t-1]})$ and $\lambda^t=\lambda^t(\xi_{[t-1]})$ are mappings of the history $\xi_{[t-1]}=(\xi_1,\ldots, \xi_{t-1})$ of the generated random process and hence are random.

Under the assumptions of weak convexity and strict feasibility, the main contributions of this paper can be summarized as follows.
\begin{itemize}
 \item[(a)] By setting the parameters in PMQSopt as $\sigma = T^{-3/4}$, $\alpha = \beta T^{1/4}$ and $\tau=T^{1/2}$, the average squared norm of the gradient of the Moreau envelope of the Lagrangian converges to stationarity with a rate of ${\rm O}(T^{-1/4})$. Furthermore, the constraint violation and complementarity violation both converge to zero at the same rate of ${\rm O}(T^{-1/4})$.
\item[(b)]  By setting the parameters in PMQSopt as $\sigma = T^{-3/4}$, $\alpha = \beta T^{1/4}$
 and $\tau=T^{1/2}$, then  the following high-probability guarantees hold:
$$
\begin{array}{l}
{\rm Pr} \left[\displaystyle \frac{1}{T}
\displaystyle \sum_{t=1}^T\|R_{\alpha/2}(x^{t},\lambda^{t}\|
\leq K_1\left(T^{-1/8}\right)\right] \geq 1-\displaystyle \frac{1}{T^{2/3}};\\[12pt]
{\rm Pr} \left[\displaystyle \frac{1}{T}\displaystyle \sum_{t=1}^Tg_i(x^t) \leq K_2\left(T^{-1/4}\right)
 \right]\geq 1-\displaystyle \frac{2}{T^{2/3}};\\[12pt]
 {\rm Pr} \left[\displaystyle \frac{1}{T}\Big|\displaystyle \sum_{t=1}^T\langle \lambda^t, g(x^t)\rangle\Big|\leq K_3\left(T^{-1/4}\right)\right] \geq 1-\displaystyle \frac{3}{T^{2/3}}.
\end{array}
$$
\item[(c)] When the constraint set is convex, we show that by solving the subproblem in PMQSopt via its dual, the algorithm can be implemented as a practical projection method for the stochastic optimization problem.
\end{itemize}

  The remainder of this paper is structured as follows. Section 2 develops the key properties of PMQSopt that are essential for analyzing Lagrangian gradient violation, constraint violation, and complementarity violation. In Section 3, we establish the sample complexity of PMQSopt in terms of the total number of stochastic gradient evaluations. Section 4 provides high-probability guarantees for the Lagrangian gradient, constraint violation, and complementarity violation. Section 5 discusses the solution of the PMQSopt subproblems and presents numerical results obtained from its implementation. Finally, Section 6 concludes the paper with a summary and a discussion of future directions.


\section{Properties of PMQSopt}\label{Sec2}
\setcounter{equation}{0}

 Let $\Phi$ be the feasible region of Problem (\ref{eq:1}):
$$
\Phi=\left\{x\in X_0: g_i(x) \leq 0, i=1,\ldots,p\right\}.
$$
We assume that expectations
$$
\mathbb{E}[F(x,\xi)]=\displaystyle \int_{\Xi} F(x,\xi)dP(\xi),\, \mathbb{E}[G_i(x,\xi)]=\displaystyle \int_{\Xi} G_i(x,\xi)dP(\xi), i=1,\ldots, p
$$
are well defined and finite valued for every $x\in {\cal O}_0$. Moreover, we assume that the expected value function $f(\cdot)$ and $g_i(\cdot)$ are continuous  on ${\cal O}_0$. Denote $G(x,\xi)=(G_1(x,\xi),\ldots, G_p(x,\xi))^T$ and $g(x)=(g_1(x),\ldots, g_p(x))^T$, then
$$
g(x)=\displaystyle \int_{\Xi} G(x,\xi)dP(\xi).
$$
We make the following assumptions about problem functions, which will be used in somewhere.
\begin{description}
\item[(A1)] It is possible to generate an independent i.i.d. sample $\xi_1,\xi_2,\ldots,$ of realizations of random vector $\xi$.
    \item[(A2)]Let $D_0>0$ such that
        $$ \|x-z\|\leq D_0, \forall x,z \in X_0.
        $$

\item[(A3)] Let  $\nu_g>0$ such that
                $$
         \|G(x,\xi)\| \leq \nu_g, \forall x \in {\cal O}_0, \xi \in \Xi.
        $$
        \item[(A4)]  Let $\kappa_f>0$ and $\kappa_g>0$ such that
                $$
        \|\nabla_x F(x,\xi)\| \leq \kappa_f, \,\, \|\nabla_x G_i(x,\xi)\| \leq \kappa_g,i=1,\ldots, p, \forall (x,\xi)\in {\cal O}_0\times \Xi.
        $$
                \item[(A5)] There exist $\epsilon_0>0$ and $\widehat x \in X_0$ such that
        $$
        g_i(\widehat x) \leq -\epsilon_0, \,\, i =1,\ldots, p.
        $$
         \item[(A6)]There are positive numbers $L_i$ ($i=0,1,\ldots, p$)  such that $F(\cdot,\xi)$ is $L_0$-weakly convex over ${\cal O}_0$ and $G_i(\cdot,\xi)$ is $L_i$-weakly convex over ${\cal O}_0$ for each $\xi \in \Xi$ and $i=0,\ldots, p$.
    \end{description}
 We also need the following assumptions about  parameters in the algorithm.
\begin{description}
\item[(B1)] Assume that $\Sigma^t_0\in \mathbb S^n$ is positively definite.
\item[(B2)] Assume that $q^t_i(x)\leq G_i(x,\xi_t)$ for $i=1,\ldots,p$.
\item[(B3)] Assume that  $\Sigma^t_i$ is negative semidefinite, $\|\Sigma^t_i\|\leq \kappa_{\Sigma}$ for $i=1,\ldots, p$, where $\kappa_{\Sigma}>0$ is some positive number.
   \item[(B4)] Assume that
$
{\cal L}^t_{\sigma }(x,\lambda^t)
$ is a convex function for every $t \in \textbf{N}$.
\end{description}
\begin{remark}\label{remarkB}
If Assumption (A6) is satisfied, then $f$ and $g_i$, $i=1,\ldots,p$ are weakly convex  over ${\cal O}_0$.
Conditions (B1)--(B4) are not restricted conditions when $F$ and $G_i$,$i=1,\ldots, p$ satisfy Assumptions (A2)--(A6). We will show, in Section \ref{Sec3},  how to construct $\Sigma^t_i, i=0,1,\ldots, p$ to satisfy
conditions (B1)--(B4).
\end{remark}
In this section, we develop properties of PMQSopt, which will be used in the analysis for Lagrange gradient, and constraint violation and complementarity violation. The following auxiliary lemma will be used several times in the sequel.
\begin{lemma}\label{lem:opt-x}
Let Assumption (A1) be satisfied. Suppose that  $\Sigma^t_0\in \mathbb S^n$ is positively definite such that Assumption (B4) holds. Then for any $z\in  X_0$, we have
\begin{equation}\label{eq:opt-x-1}
\begin{array}{ll}
\displaystyle\langle \nabla_x F(x^k,\xi_k),x^{k+1}-x^k \rangle  +\displaystyle \frac{1}{2}\langle \Sigma^k_0(x^{k+1}-x^k),x^{k+1}-x^k\rangle+ \frac{1}{2\sigma}\|\lambda^{k+1}\|^2
+ \frac{\alpha}{2} \|x^{k+1}-x^k\|^2  \\[5pt]
\leq \displaystyle\langle \nabla_x F(x^k,\xi_k),z-x^k \rangle +\displaystyle \frac{1}{2}\langle \Sigma^k_0(z-x^k),z-x^k\rangle \\[15pt]
\quad\quad + \displaystyle\frac{1}{2\sigma}\left[\displaystyle \sum_{i=1}^p[\lambda^k_i+\sigma (G_i(x^k,\xi_k)+\langle \nabla_x G_i(x^k,\xi_k), z-x^k \rangle)+\displaystyle \frac{1}{2}\langle \Sigma^k_i(z-x^k),z-x^k\rangle]_+^2\right]\\[15pt]
\quad\quad+ \displaystyle\frac{\alpha}{2}(\|z-x^k\|^2-\|z-x^{k+1}\|^2).
\end{array}
\end{equation}
In particular, if we take $z=x^k$, it yields
\begin{equation}\label{eq:opt-x-2}
\begin{array}{ll}
\displaystyle\langle \nabla_x F(x^k,\xi_k),x^{k+1}-x^k \rangle + \frac{1}{2\sigma}\|\lambda^{k+1}\|^2
+ \alpha \|x^{k+1}-x^k\|^2+\displaystyle \frac{1}{2}\|x^{k+1}-x^k\|^2_{\Sigma^t_0}\\[10pt]
\leq  \displaystyle\frac{1}{2\sigma}\left[ \sum_{i=1}^p[\lambda^k_i+\sigma G_i(x^k,\xi_k)]_+^2\right].
\end{array}
\end{equation}
\end{lemma}
{\bf Proof}.
Noting that the minimization problem for defining $x^{k+1}$ in (\ref{xna}) is a strongly convex optimization problem and in view of its optimality conditions, we have that $x^{k+1}$ is also the optimal solution to the following problem
\[
\begin{array}{ll}
\min\limits_{x \in  X_0}\quad  q^k_0(x)+  \displaystyle\frac{1}{2\sigma} \displaystyle\sum_{i=1}^p\left[\lambda_i^k+\sigma q^k_i(x)\right]_+^2+\displaystyle \frac{\alpha}{2}(\|x-x^k\|^2-\|x-x^{k+1}\|^2).
\end{array}
\]
Then,
the claimed results are obvious.
\hfill $\Box$

In order to give a bound for $\sum_{t=1}^T G_i(x^t,\xi_t)$, we need to estimate an upper bound of $\|x^{t+1}-x^t\|$, which is given in the following lemma.

\begin{lemma}\label{lem:3}
Suppose  Assumptions (A1)-(A4) and (B1)-(B4) hold.
If $2\alpha-p(\kappa_g+\kappa_{\Sigma}D_0/2)^2\sigma>0$, then
\begin{equation}\label{eq:diffX}
\begin{array}{ll}
\|x^{t+1}-x^t\|
&\leq \displaystyle\frac{2}{2\alpha+\lambda_{\min}\left(\Sigma^t_0\right)-p(\kappa_g+\kappa_{\Sigma}D_0/2)^2\sigma}
\left(\kappa_f+(\kappa_g+\kappa_{\Sigma}D_0/2)\sqrt{p}
[\|\lambda^t\|+\nu_g\sigma]\right).
\end{array}
\end{equation}
If
$\alpha-p(\kappa_g+\kappa_{\Sigma}D_0/2)^2\sigma>0$, then
\begin{equation}\label{eq:diffXa}
\|x^{t+1}-x^t\|\leq
\displaystyle\frac{2}{\alpha+\lambda_{\min}\left(\Sigma^t_0\right)}\left(\kappa_f+(\kappa_g+\kappa_{\Sigma}D_0/2)\sqrt{p}
[\|\lambda^t\|+\nu_g\sigma]\right).
\end{equation}
\end{lemma}
{\bf Proof}. Since $x^{t+1}$ is a solution to Problem (\ref{xna}), we have from (\ref{eq:opt-x-2}) that
$$
\begin{array}{l}
\langle \nabla_x F(x^t,\xi_t), x^{t+1}-x^t\rangle +\displaystyle \frac{1}{2} \left \langle \Sigma^t_0(x^{t+1}-x^t),x^{t+1}-x^t\right \rangle+\displaystyle \frac{1}{2\sigma}\|\lambda^{t+1}\|^2 +\alpha\|x^{t+1}-x^t\|^2\\[10pt]
 \leq \displaystyle \frac{1}{2\sigma}\|[\lambda^{t} +\sigma q^t(x^t)]_+\|^2
 \leq \displaystyle \frac{1}{2\sigma}\|[\lambda^{t} +\sigma G(x^t,\xi_t)]_+\|^2,
\end{array}
$$
 we have
\[
\left (\alpha+\displaystyle \frac{1}{2}\lambda_{\min}\left(\Sigma^t_0\right)\right)\|x^{t+1}-x^t\|^2\leq  \kappa_f\|x^{t+1}-x^t\|+
\frac{1}{2\sigma}\sum_{i=1}^p\left([a_i]_+^2-[b_i]_+^2\right),
\]
in which, for simplicity, we use
\[
a_i:=\lambda_i^t+\sigma G_i(x^t,\xi_t),\quad b_i:=\lambda_i^t+\sigma (G_i(x^t,\xi_t)+\langle \nabla_x G_i(x^t,\xi_t),(x^{t+1}-x^t)\rangle+\displaystyle \frac{1}{2}\langle \Sigma^t_i(x^{t+1}-x^t),x^{t+1}-x^t\rangle).
\]
Noticing that
\[
\begin{array}{ll}
[a_i]_+^2-[b_i]_+^2&=([a_i]_++[b_i]_+)([a_i]_+-[b_i]_+)\leq (|a_i|+|b_i|)\cdot|a_i-b_i|\\[8pt]
&\leq (2|a_i|+|b_i-a_i|)\cdot|a_i-b_i|=2|a_i|\cdot|a_i-b_i|+(a_i-b_i)^2\\[8pt]
&\leq 2|\lambda_i^t+\sigma G_i(x^t,\xi_t)|\cdot\sigma\left(\kappa_g\|x^{t+1}-x^t\|+\displaystyle \frac{1}{2}\|\Sigma^t_i\|
\|x^{t+1}-x^t\|^2\right)\\[8pt]
&\quad +\sigma^2\left(\kappa_g\|x^{t+1}-x^t\|+\displaystyle \frac{1}{2}\|\Sigma^t_i\|
\|x^{t+1}-x^t\|^2\right)^2,
\end{array}
\]
we obtain
\[
\left(2\alpha+\lambda_{\min}\left(\Sigma^t_0\right)\right)\|x^{t+1}-x^t\|\leq 2\kappa_f+\sum_{i=1}^p\left[(2\kappa_g+\kappa_{\Sigma}D_0)|\lambda_i^t+\sigma G_i(x^t,\xi_t)|+\sigma(\kappa_g+\kappa_{\Sigma}D_0/2)^2\|x^{t+1}-x^t\|\right].
\]
If $2\alpha-p(\kappa_g+\kappa_{\Sigma}D_0/2)^2\sigma>0$, it yields
\[
\begin{array}{ll}
\|x^{t+1}-x^t\|&\leq \displaystyle \frac{2}{2\alpha+\lambda_{\min}\left(\Sigma^t_0\right)-p(\kappa_g+\kappa_{\Sigma}D_0/2)^2\sigma}
\left(\kappa_f+(\kappa_g+\kappa_{\Sigma}D_0/2)\sum_{i=1}^p|\lambda_i^t+\sigma G_i(x^t,\xi_t)|\right) \\[12pt]
&\leq \displaystyle\frac{2}{2\alpha+\lambda_{\min}\left(\Sigma^t_0\right)-p(\kappa_g+\kappa_{\Sigma}D_0/2)^2\sigma}
\left(\kappa_f+(\kappa_g+\kappa_{\Sigma}D_0/2)\sqrt{p}
[\|\lambda^t\|+\nu_g\sigma]\right),
\end{array}
\]
namely (\ref{eq:diffX}) is satisfied,
where the last inequality is obtained from the facts that $\sum_{i=1}^p|\lambda_i^t|\leq\sqrt{p}\|\lambda^t\|$ and $\sum_{i=1}^p|G_i(x^t,\xi_t)|\leq\sqrt{p}\|G(x^t,\xi_t)\|\leq\sqrt{p}\nu_g$.
\hfill $\Box$

\begin{lemma}\label{lem:1}
Let $(x^t,\lambda^t)$ be generated by PMQSopt,  Assumptions (A1) -- (A4), (B2) and  (B3) be satisfied. Then
\begin{equation}\label{eq:lambda2}
\|\lambda^{t+1}\|^2 \leq \|\lambda^t\|^2 +2 \sigma \langle \lambda^t, G(x^{t+1}, \xi_t)\rangle+\sigma^2 \nu_g^2,
\end{equation}
and
\begin{equation}\label{eq:2lambda}
\|\lambda^t\|-\gamma_1\sigma \leq \|\lambda^{t+1}\| \leq \|\lambda^t\|+ \gamma_1\sigma.
\end{equation}
where
$$
\gamma_1=\nu_g+\sqrt{p} \left(\kappa_gD_0+\displaystyle \frac{1}{2}\kappa_{\Sigma}D_0^2\right).
$$
Moreover, one has
\begin{equation}\label{eq:lambdainorm}
|\lambda^{t+1}_i-\lambda^t_i| \leq \sigma \left[\nu_g+\left(\kappa_g+\displaystyle \frac{\kappa_{\Sigma}D_0}{2}\right)\|x^{t+1}-x^t\|\right]\leq
\gamma_2\sigma,
\end{equation}
where
$$
\gamma_2=\left[\nu_g+\left(\kappa_gD_0+\displaystyle \frac{\kappa_{\Sigma}D_0^2}{2}\right)\right].
$$
\end{lemma}
{\bf Proof}. Noting that for any $a \in \mathbb R$, $[a]_+^2 \leq a^2$, we have from Assumptions (B2) and (A3) that
$$
\begin{array}{ll}
\|\lambda^{t+1}\|^2 &=\displaystyle \sum_{i=1}^p [\lambda_i^t+\sigma q^t_i(x^{t+1})]_+^2 \leq
\displaystyle \sum_{i=1}^p [\lambda^t_i+\sigma G_i(x^{t+1},\xi_t)]_+^2\leq \displaystyle \sum_{i=1}^p [\lambda^t_i+\sigma G_i(x^{t+1},\xi_t)]^2\\[6pt]
&= \displaystyle \sum_{i=1}^p\left( [\lambda^t_i]^2++2 \sigma \langle \lambda^t, G(x^{t+1}, \xi_t)\rangle+\sigma^2  G_i(x^{t+1},\xi_t)^2\right)\leq \|\lambda^t\|^2 +2 \sigma \langle \lambda^t, G(x^{t+1}, \xi_t)\rangle+\sigma^2 \nu_g^2.
\end{array}
$$
It follows from the nonexpansion property of the projection $\Pi_{\mathbb R^p_+}(\cdot)$, we have from Assumptions (A2) -- (A4) and (B3)  that
$$
\begin{array}{ll}
\|\lambda^{t+1}-\lambda^t\| & =\|[\lambda^t+\sigma q^t(x^{t+1})]_+-[\lambda^t]_+\|\\[6pt]
&  \leq \sigma\|G(x^t,\xi_t)\|+\sigma \left (\displaystyle \sum_{i=1}^p \left(\|\nabla_x G_i(x^t,\xi_t)\|\|x^{t+1}-x^t\|+\displaystyle \frac{1}{2}\|\Sigma^t_i\|\|x^{t+1}-x^t\|^2\right)^2\right)^{1/2}\\[8pt]
& \leq \sigma\nu_g+\sigma \left (\displaystyle \sum_{i=1}^p \left(\kappa_gD_0+\displaystyle \frac{1}{2}\kappa_{\Sigma}D_0^2\right)^2\right)^{1/2} \leq \sigma\left[\nu_g+\sqrt{p} \left(\kappa_gD_0+\displaystyle \frac{1}{2}\kappa_{\Sigma}D_0^2\right)\right],
\end{array}
$$
which implies (\ref{eq:2lambda}) from the definition of $\gamma_1$. Inequality
(\ref{eq:lambdainorm}) can be obtained from the definition of $\lambda^{t+1}_i$ and Assumptions (A3), (A4) and  (B3).
The proof is completed. \hfill $\Box$\\

The following lemma is just Lemma 6 of \cite{Zhang2023}, which will be used to establishing the boundedness of $\mathbb E\|\lambda^t\|$.
\begin{lemma}\label{lem:4}
Let $(x^t,\lambda^t)$ be generated by PMQSopt and Assumptions (A1) and  (A5) be satisfied. Then for any $t_2 \leq t_1-1$ where $t_1$ and $t_2$ are positive integers,
\begin{equation}\label{eq:slaterCondlamd}
\displaystyle \mathbb{E}\left[\langle \lambda^{t_1},G(\widehat x, \xi_{t_1}) \rangle \,|\, \xi_{[t_2]}\right]
\leq -\epsilon_0 \mathbb{E}\left[\|\lambda^{t_1}\| \,|\, \xi_{[t_2]}\right].
\end{equation}
\end{lemma}

The following lemma is similar to Lemma 7 of \cite{Zhang2023}, however the assumptions here are different from those in  \cite{Zhang2023}.
\begin{lemma}\label{lem:l5}
 Let $s > 0$ be an arbitrary integer. Let  Assumptions (A1) -- (A4) and (B1)-- (B4) be satisfied.
  Let
  \begin{equation}\label{eq:Sigma0c}
  \Sigma^t_0=-\displaystyle\sum_{i=1}^p \lambda^t_i \Sigma^t_i+\tau I
  \end{equation}
  for some $\tau >0$.
    Assume $\epsilon_0>0$ and  $\kappa_\Sigma>0$ satisfy
  \begin{equation}\label{extrCon}
  \sqrt{p}\kappa_\Sigma D_0^2\leq \epsilon_0.
  \end{equation}
  At each round  $t \in \{1,2,\ldots\}$ in  PMQSopt. For  any $\alpha > 2\kappa_{\Sigma}$ and
 \begin{equation}\label{eq:theta9}
 \vartheta (\sigma,\alpha,\tau, s)=\displaystyle \frac{\epsilon_0\sigma s}{4}+\gamma_1\sigma(s-1)+\displaystyle \frac{2(\alpha+\tau) D_0^2}{\epsilon_0s}+\displaystyle \frac{4\kappa_fD_0}{\epsilon_0}+
 \displaystyle \frac{2\sigma \nu_g^2}{\epsilon_0},
 \end{equation}
 the following  holds
\begin{equation}\label{eq:6}
|\|\lambda^{t+1}\|-\|\lambda^t\||\leq \sigma \gamma_1
\end{equation}
and
\begin{equation}\label{eq:7}
\mathbb{E}\left [ \|\lambda^{t+s}\|-\|\lambda^t\| \,|\, \xi_{[t-1]}\right]
\leq \left
\{
\begin{array}{ll}
s \sigma \gamma_1 & \mbox{if } \|\lambda^t\| < \vartheta (\sigma,\alpha,\tau, s),\\[6pt]
-s \displaystyle \frac{\sigma \epsilon_0}{4} & \mbox{if } \|\lambda^t\| \geq  \vartheta (\sigma,\alpha,\tau,s).
\end{array}
\right.
\end{equation}
\end{lemma}
{\bf Proof}. Inequality (\ref{eq:6}) follows from Lemma \ref{lem:1}. We only need to establish (\ref{eq:7}).
Since it is obvious that
$$
 \mathbb{E}\left[\|\lambda^{t+s}\|-\|\lambda^t\|\,|\,\xi_{[t-1]}\right]\leq s \sigma \gamma_1
$$
 when $\|\lambda^t\|< \vartheta (\sigma,\alpha,s)$, it remains to prove
$$ \mathbb{E}\left [ \|\lambda^{t+s}\|-\|\lambda^t\| \,|\, \xi_{[t-1]}\right]
\leq -s \displaystyle \frac{\sigma \epsilon_0}{4}
$$
when $\|\lambda^t\| \geq  \vartheta (\sigma,\alpha,\tau, s)$.

For given positive integer $s$, suppose $\|\lambda^t\| \geq  \vartheta (\sigma,\alpha,\tau, s)$. For any $l \in \{t,t+1,\ldots, t+s-1\}$, since from the choice of $\Sigma^t_0$, Assumption (B4) holds, and ${\cal L}^l_{\sigma }(x,\lambda^l) +\displaystyle \frac{\alpha}{2}\|x-x^l\|^2$ is strongly convex with modulus $\displaystyle \frac{\alpha+\tau}{2}$, one has from Assumption (B2) for
$\widetilde \Sigma^t_0=\Sigma^t_0-\tau I$, that
$$
\begin{array}{l}
\langle \nabla_x F(x^l,\xi_l),x^{l+1}-x^l \rangle+\displaystyle \frac{1}{2}\left \langle \widetilde \Sigma^l_0(x^{l+1}-x^l), x^{l+1}-x^l \right \rangle+\displaystyle \frac{1}{2\sigma}\|\lambda^{l+1}\|^2+\displaystyle \frac{\alpha}{2}\|x^{l+1}-x^l\|^2\\[10pt]
\quad \,\leq \langle \nabla_x F(x^l,\xi_l),\widehat x-x^l \rangle+\displaystyle \frac{1}{2}\left \langle \widetilde \Sigma^l_0(\widehat x-x^l), \widehat x-x^l \right \rangle+\displaystyle \frac{1}{2\sigma}\|[\lambda^{l}+\sigma q^l(\widehat x)]_{+}\|^2\\[10pt]
\quad \quad \quad +\displaystyle \frac{\alpha+\tau}{2}\left[\|\widehat x-x^l\|^2-\|\widehat x-x^{l+1}\|^2\right]\\[10pt]
\quad \,\leq \langle \nabla_x F(x^l,\xi_l),\widehat x-x^l \rangle+\displaystyle \frac{1}{2}\left \langle \widetilde \Sigma^l_0(\widehat x-x^l), \widehat x-x^l \right \rangle+\displaystyle \frac{1}{2\sigma}\|[\lambda^{l}+\sigma G(\widehat x,\xi_l)]_{+}\|^2\\[10pt]
\quad \quad \quad +\displaystyle \frac{\alpha+\tau}{2}\left[\|\widehat x-x^l\|^2-\|\widehat x-x^{l+1}\|^2\right].
\end{array}
$$
Using Assumption (A3) and the following inequality
$$
\|[\lambda^{l}+\sigma G(\widehat x,\xi_l)]_{+}\|^2 \leq \|\lambda^l\|^2+2\sigma \langle \lambda^l,G(\widehat x,\xi_l)\rangle
+\sigma^2\|G(\widehat x,\xi_l)\|^2,
$$
we obtain from Assumptions (A2) -- (A4) and (B3) that
\begin{equation}\label{eq:8}
\begin{array}{l}
\displaystyle \frac{1}{2\sigma} \left[\|\lambda^{l+1}\|^2-\|\lambda^l\|^2\right]
\leq \langle \nabla_x F(x^l,\xi_l),\widehat x-x^{l+1}\rangle +\displaystyle \frac{1}{2}\left \langle \widetilde \Sigma^l_0(\widehat x-x^l), \widehat x-x^l \right \rangle\\[10pt]
 \quad \quad +\displaystyle \frac{1}{2\sigma}\left[\|[\lambda^{l}+\sigma G(\widehat x,\xi_l)]_{+}\|^2-\|\lambda^l\|^2\right]+\displaystyle \frac{\alpha+\tau}{2}\left[\|\widehat x-x^l\|^2-\|\widehat x-x^{l+1}\|^2\right]\\[10pt]
\quad \quad -\left[\displaystyle \frac{\alpha+\tau}{2}\|x^{l+1}-x^l\|^2+\displaystyle \frac{1}{2}\left \langle \widetilde \Sigma^l_0(x^{l+1}-x^l), x^{l+1}-x^l \right \rangle\right]\\[10pt]
\leq \kappa_fD_0+ \displaystyle \frac{1}{2}D_0^2\displaystyle \sum_{i=1}^p\lambda^l_i \|\Sigma^l_i\|+
 \langle \lambda^l,G(\widehat x,\xi_l)\rangle
+\displaystyle \frac{\sigma}{2}\|G(\widehat x,\xi_l)\|^2\\[10pt]
\quad \quad -\displaystyle \frac{\alpha+\tau}{2}\|x^{l+1}-x^l\|^2+\displaystyle \frac{\alpha+\tau}{2}\left[\|\widehat x-x^l\|^2-\|\widehat x-x^{l+1}\|^2\right]\\[10pt]
\leq \kappa_fD_0+\displaystyle \frac{1}{2}\sqrt{p}\kappa_{\Sigma}D_0^2\|\lambda^l\|+ \langle \lambda^l,G(\widehat x,\xi_l)\rangle
+\displaystyle \frac{\sigma}{2}\nu_g^2\\[10pt]
\quad \quad -\displaystyle \frac{\alpha+\tau}{2}\|x^{l+1}-x^l\|^2+\displaystyle \frac{\alpha+\tau}{2}\left[\|\widehat x-x^l\|^2-\|\widehat x-x^{l+1}\|^2\right]\\[10pt]
\leq \kappa_f+\displaystyle \frac{1}{2}\varepsilon_0\|\lambda^l\|+ \langle \lambda^l,G(\widehat x,\xi_l)\rangle
+\displaystyle \frac{\sigma}{2}\nu_g^2\\[10pt]
\quad \quad -\displaystyle \frac{\alpha+\tau}{2}\|x^{l+1}-x^l\|^2+\displaystyle \frac{\alpha+\tau}{2}\left[\|\widehat x-x^l\|^2-\|\widehat x-x^{l+1}\|^2\right]
\end{array}
\end{equation}
Making  a summation of (\ref{eq:8}) over $\{t,t+1, t+s-1\}$ and taking conditional expectation on $\xi_{[t-1]}$ , we obtain from
Lemma \ref{lem:4} that
\begin{equation}\label{eq:9}
\begin{array}{l}
\displaystyle \frac{1}{2\sigma} \mathbb{E}\left[\|\lambda^{t+s}\|^2-\|\lambda^t\|^2\,|\, \xi_{[t-1]}\right]
\leq \kappa_f D_0 s + \displaystyle \frac{\sigma}{2}\nu_g^2s +\displaystyle \sum_{l=t}^{t+s-1} \mathbb{E}\left[\langle \lambda^l,G(\widehat x,\xi_l)\rangle+\displaystyle \frac{1}{2} \varepsilon_0 \|\lambda^l\|\,|\, \xi_{[t-1]}\right]
\\[10pt]
\quad \quad +\displaystyle \frac{\alpha+\tau}{2}\mathbb{E}\left[\left(\|\widehat x-x^t\|^2-\|\widehat x-x^{t+s}\|^2\right)\,|\, \xi_{[t-1]}\right]\\[10pt]
\leq \kappa_f D_0s + \displaystyle \frac{\sigma}{2}\nu_g^2s -\displaystyle \frac{1}{2}\epsilon_0\displaystyle \sum_{l=0}^{s-1} \mathbb{E}\left[\|\lambda^{t+l}\|\,|\, \xi_{[t-1]}\right]
+\displaystyle \frac{\alpha+\tau}{2}\mathbb{E}\left[\left(\|\widehat x-x^t\|^2-\|\widehat x-x^{t+s}\|^2\right)\,|\, \xi_{[t-1]}\right]\\[10pt]
\leq \kappa_f D_0 s + \displaystyle \frac{\sigma}{2}\nu_g^2s -\displaystyle \frac{1}{2}\epsilon_0\displaystyle \sum_{l=0}^{s-1} \mathbb{E}\left[\|\lambda^{t}\|-\sigma\gamma_1l \,|\, \xi_{[t-1]}\right]
\\[10pt]
\quad \quad +\displaystyle \frac{\alpha+\tau}{2}\mathbb{E}\left[\left(\|\widehat x-x^t\|^2-\|\widehat x-x^{t+s}\|^2\right)\,|\, \xi_{[t-1]}\right] \quad (\mbox{from } \|\lambda^{t+1}\|\geq \|\lambda^t\|-\sigma \gamma_1)\\[8pt]
\leq \kappa_f D_0 s + \displaystyle \frac{\sigma}{2}\nu_g^2s +\displaystyle \frac{\alpha+\tau}{2}\mathbb{E}\left[\left(\|\widehat x-x^t\|^2-\|\widehat x-x^{t+s}\|^2\right)\,|\, \xi_{[t-1]}\right]\\[10pt]
 \quad \quad + \epsilon_0\sigma \gamma_1 \displaystyle \frac{s(s-1)}{4}
-\displaystyle \frac{1}{2}\epsilon_0\displaystyle \sum_{l=0}^{s-1} \mathbb{E}\left[\|\lambda^{t}\| \,|\, \xi_{[t-1]}\right]
\end{array}
\end{equation}
From (\ref{eq:9}), we get that
\begin{equation}\label{eq:10}
\begin{array}{l}
\mathbb{E}\left[\|\lambda^{t+s}\|^2\,|\, \xi_{[t-1]}\right]\leq
\mathbb{E}\left[\|\lambda^t\|^2\,|\, \xi_{[t-1]}\right]\\[10pt]
\quad \quad +2 \sigma \kappa_f D_0 s+\sigma^2\nu_g^2s+(\alpha+\tau) \sigma D_0^2
+0.5\epsilon_0\sigma^2\gamma_1s(s-1)-\epsilon_0\sigma s \mathbb{E}\left[\|\lambda^{t}\| \,|\, \xi_{[t-1]}\right]\\[10pt]
=\mathbb{E}\left[(\|\lambda^t\|-\displaystyle \frac{\epsilon_0\sigma}{4}s)^2\,|\, \xi_{[t-1]}\right]
-\displaystyle \frac{\epsilon_0^2\sigma^2}{16}s^2+0.5\epsilon_0\sigma^2 \gamma_1s(s-1)\\[10pt]
\quad \quad + (\alpha+\tau) \sigma D_0^2+2\sigma \kappa_f D_0s+\sigma^2\nu_g^2s-\displaystyle \frac{1}{2}\epsilon_0\sigma s \mathbb{E}\left[\|\lambda^{t}\| \,|\, \xi_{[t-1]}\right]\\[10pt]
\leq \mathbb{E}\left[(\|\lambda^t\|-\displaystyle \frac{\epsilon_0\sigma}{4}s)^2\,|\, \xi_{[t-1]}\right]
-\displaystyle \frac{3\epsilon_0^2\sigma^2}{16}s^2\\[10pt]
+\left( \displaystyle \frac{\epsilon_0^2\sigma^2}{8}s^2+0.5\epsilon_0\sigma^2 \gamma_1s(s-1)+ (\alpha +\tau) \sigma D_0^2+2\sigma \kappa_f D_0 s+\sigma^2\nu_g^2s-\displaystyle \frac{1}{2}\epsilon_0\sigma s \vartheta (\sigma,\alpha,\tau,s)\right)\\[10pt]
= \mathbb{E}\left[(\|\lambda^t\|-\displaystyle \frac{\epsilon_0\sigma}{4}s)^2\,|\, \xi_{[t-1]}\right]-\displaystyle \frac{3\epsilon_0^2\sigma^2}{16}s^2
\leq \mathbb{E}\left[(\|\lambda^t\|-\displaystyle \frac{\epsilon_0\sigma}{4}s)^2\,|\, \xi_{[t-1]}\right]
=(\|\lambda^t\|-\displaystyle \frac{\epsilon_0\sigma}{4}s)^2.
\end{array}
\end{equation}
Taking square root on both sides yields
$$
\sqrt{\mathbb{E}\left[\|\lambda^{t+s}\|^2\,|\, \xi_{[t-1]}\right]}\leq \|\lambda^t\|-\displaystyle \frac{\epsilon_0\sigma}{4}s.
$$
By the concavity of function
$\sqrt{x}$ and Jensen's inequality, we have
$$
\mathbb{E}\left[\|\lambda^{t+s}\|\,|\, \xi_{[t-1]}\right]\leq \sqrt{\mathbb{E}\left[\|\lambda^{t+s}\|^2\,|\, \xi_{[t-1]}\right]}.
$$
This implies that
$$
\mathbb{E}\left[\|\lambda^{t+s}\|\,|\, \xi_{[t-1]}\right]\leq
\|\lambda^t\|-\displaystyle \frac{\epsilon_0\sigma}{4}s.
$$
The proof is completed. \hfill $\Box$

The following lemmas come from Yu et. al \cite{Yu2017}, which can be used to deal with the random process $\{\|\lambda^t\|\}$ and probability analysis for objective regret and constraint violation regret, respectively.
\begin{lemma}\label{lem:l7}
Let $\{Z(t), t \geq 0\}$ be a discrete time stochastic process adapted to a filtration $\{{\cal F}(t), t\geq
0\}$ with $Z(0) = 0$ and ${\cal F}(0) = \{\emptyset, \Omega\}$. Suppose there exist an integer $t_0 >0$, real constants $\theta>0$, $\delta_{\max}>0$ and $ 0 <\zeta \leq \delta_{\max}$ such that
\begin{equation}\label{eq:Y9}
\begin{array}{rl}
|Z(t+1)-Z(t)| & \leq \delta_{\max}\mbox{ and }\\[12pt]
\mathbb{E}[Z(t+t_0)-Z(t)\,|\, {\cal F}(t)] & \leq \left
\{
\begin{array}{ll}
t_0 \delta_{\max} & \mbox{if } Z(t) < \theta\\[6pt]
-t_0\zeta & \mbox{if } Z(t) \geq \theta
\end{array}
\right.
\end{array}
\end{equation}
hold for all $t \in \{1,2,\ldots\}.$ Then the following properties are satisfied.
\begin{itemize}
\item[{\rm (1)}] The  following inequality holds
$$\mathbb{E}[Z(t)] \leq \theta +t_0 \delta_{\max}+t_0 \displaystyle \frac{4 \delta_{\max}^2}{\zeta}\log \left[ \displaystyle \frac{8 \delta_{\max}^2}{\zeta^2} \right], \forall t \in \{1,2,\ldots\}.$$
    \item[{\rm (2)}] For any constant $0 < \mu <1$, we have
    $$
    {\rm Pr}\left\{Z(t)\geq z\right\} \leq \mu, \forall t \in \{1,2,\ldots\},
    $$
    where
    $$
    z=\theta +t_0 \delta_{\max}+t_0 \displaystyle \frac{4 \delta_{\max}^2}{\zeta}\log \left[ \displaystyle \frac{8 \delta_{\max}^2}{\zeta^2}\right]+t_0 \displaystyle \frac{4 \delta_{\max}^2}{\zeta}\log\left(\displaystyle \frac{1}{\mu} \right).
    $$
\end{itemize}
\end{lemma}

In order to use Lemma \ref{lem:l7}  to analyze proerties of PMQSopt for Problem (\ref{eq:1}), we introduce the following notations.
Let  $\theta=\vartheta (\sigma,\alpha,\tau,s)$, $\delta_{\max}=\sigma \gamma_1$ and $\zeta =\displaystyle \frac{\sigma}{4}\epsilon_0$, and $t_0=s$, and
define
$$
\psi(\sigma,\alpha,\tau, s)=\theta +t_0 \delta_{\max}+t_0 \displaystyle \frac{4 \delta_{\max}^2}{\zeta}\log \left[ \displaystyle \frac{8 \delta_{\max}^2}{\zeta^2} \right]
$$
 and
$$
\phi (\sigma,\alpha,\tau, s,\mu)=\psi(\sigma,\alpha,\tau,s)+16 \displaystyle \frac{\gamma_1^2}{\varepsilon_0} \log \left( \displaystyle \frac{1}{\mu}\right)\sigma s.
$$
  Then $\psi(\sigma,\alpha,\tau,s)$ is expressed as
$$
\begin{array}{ll}
\psi(\sigma,\alpha,\tau,s)= &
\vartheta (\sigma,\alpha,\tau,s)+  \left[\gamma_1+\displaystyle \frac{16\gamma_1^2}{\epsilon_0}\log\displaystyle \frac{128\gamma_1^2}{\epsilon_0^2}\right] \sigma s\\[12pt]
&=\kappa_0+\kappa_1\displaystyle \frac{\alpha+\tau}{s}+ \kappa_2 s+\kappa_3
 \sigma+\kappa_4 \sigma s
 \end{array}
$$
  where
\begin{equation}\label{eq:notations}
\begin{array}{l}
\kappa_0=\displaystyle \frac{4\kappa_fD_0}{\epsilon_0},\,\,
\kappa_1=\displaystyle \frac{2D_0^2}{\epsilon_0},\,\,
\kappa_2=0,\\[10pt]
\kappa_3=\displaystyle \frac{2 \nu_g^2}{\epsilon_0}-\gamma_1,\,\,
\kappa_4=\left[2\gamma_1+\displaystyle \frac{\epsilon_0}{4}+\displaystyle \frac{16\gamma_1^2}{\epsilon_0}\log \displaystyle \frac{128\gamma_1^2}{\epsilon_0^2}\right].
\end{array}
\end{equation}
\begin{lemma}\label{lem:ebound}
 Let   Assumptions (A1) -- (A4) and (B1) -- (B4) be satisfied
 Let
  \begin{equation}\label{eq:Sigma0c1}
  \Sigma^t_0=-\displaystyle\sum_{i=1}^p \lambda^t_i \Sigma^t_i+\tau I
  \end{equation}
  for some $\tau >0$.
    Assume $\epsilon_0>0$ and  $\kappa_\Sigma>0$ satisfy
  \begin{equation}\label{extrCon1}
  \sqrt{p}\kappa_\Sigma \leq \epsilon_0.
  \end{equation}
 Let $s>0$
be an arbitrary integer. Then, it holds that
\begin{equation}\label{eq:bd1}
\mathbb E[\|\lambda^k\|]\leq \psi (\sigma,\alpha,\tau,s)
\end{equation}
Moreover, for any constant $0< \mu < 1$, we have
\begin{equation}\label{eq:bd2}
{\rm Pr}[\|\lambda^k\|\geq \phi (\sigma,\alpha,s,\tau,\mu)
] \leq \mu.
\end{equation}

If we choose
$$
\sigma=T^{-3/4},
\alpha=\beta T^{1/4},\tau=T^{1/2}, s=T^{1/2}
$$
then
$$
\psi(\sigma,\alpha,\tau, s)=\psi(T^{-3/4},\beta T^{1/4}, T^{1/2},T^{1/2})=\kappa_0+\kappa_1\displaystyle \frac{\alpha+\tau}{s}+ \kappa_2 s+\kappa_3
 \sigma+\kappa_4 \sigma s=\kappa_0+\kappa_1+(\kappa_1\beta+\kappa_4) T^{-1/4}+\kappa_3 T^{-3/4}.
$$
In this case, we obtain, if $T \geq \beta^2$, then  for $k=1,\ldots, T$,
\begin{equation}\label{eq:lamB}
\mathbb E[\|\lambda^k\|] \leq \psi(T^{-3/4},\beta T^{1/4},T^{1/2}, T^{1/2})\leq \gamma_3,
\end{equation}
where $\gamma_3=\kappa_0+\kappa_1+\kappa_3 +\kappa_4$.
In this case, for $\mu \in (0,1)$,
\begin{equation}\label{eq:bd2s}
{\rm Pr}\left[\|\lambda^k\|\geq \gamma_3+16 \displaystyle \frac{\gamma_1^2}{\varepsilon_0} \log \left( \displaystyle \frac{1}{\mu}\right)T^{-1/4}
\right] \leq \mu.
\end{equation}
\end{lemma}

Define
$$
 z^t_{\theta}=\psi (\sigma,\alpha,\tau,s)+16 \displaystyle \frac{\gamma_1^2}{\varepsilon_0} \sigma st^{\theta}.
$$
Then we have the following results, which plays an important role for estimating Lagrange gradient.
\begin{lemma}\label{lem:pdsa}

Let   the assumptions  in Lemma \ref{lem:ebound} be satisfied.
For any $\theta>0$, one has for any integer $t \geq 0$,
\begin{equation}\label{eq:pda}
\mathbb{E}\big[ \|\lambda^{t+1}\|^2 \mid \xi_{[t-1]} \big]
\le (z^t_\theta)^2 + 2\sigma\gamma_1 z^t_\theta + (\sigma\gamma_1)^2 + (t+1)^2 (\sigma\gamma_1)^2 e^{-t^\theta}
\end{equation}
and
\begin{equation}\label{eq:sumpda}
\begin{array}{ll}
\displaystyle \sum_{t=1}^T &\mathbb{E}\big[ \|\lambda^{t+1}\|^2 \mid \xi_{[t-1]} \big] \\[6pt]
&\le T\psi(\sigma,\alpha,\tau,s)^2 + 32\psi(\sigma,\alpha,\tau,s)\displaystyle \frac{\gamma_1^2}{\varepsilon_0}\sigma s \frac{T^{\theta+1}}{\theta+1}
+ 256\displaystyle \frac{\gamma_1^4}{\varepsilon_0^2}\sigma^2 s^2 \displaystyle \frac{T^{2\theta+1}}{2\theta+1} \\[12pt]
&\quad + 2\sigma\nu_g \left[ T\psi(\sigma,\alpha,\tau,s) + 16\displaystyle \frac{\gamma_1^2}{\varepsilon_0}\sigma s \frac{T^{\theta+1}}{\theta+1} \right]
+ T(\sigma\nu_g)^2 \\[12pt]
&\quad + \displaystyle \frac{(\sigma\nu_g)^2}{\theta} \left[ \Gamma\Big(\displaystyle \frac{3}{\theta},1\Big) + 2\Gamma\Big(\displaystyle \frac{2}{\theta},1\Big) + \Gamma\Big(\displaystyle \frac{1}{\theta},1\Big) \right].
\end{array}
\end{equation}
where $\Gamma (a,b)$ is the following Gamma function
$$
\Gamma (a,b)=\displaystyle \int_b^{+\infty} t^{a-1}e^{-t}dt.
$$
\end{lemma}
{\bf Proof}. In view of (\ref{eq:bd2}) of  Lemma \ref{lem:ebound}, for any constant $0< \mu < 1$, we have
$$
{\rm Pr}[\|\lambda^t\|^2\geq [\phi (\sigma,\alpha,s,\tau,\mu)]^2
] \leq \mu.
$$
which implies for $z^t_{\theta}$ that
\begin{equation}\label{eq;help00a}
{\rm Pr}[\|\lambda^t\|^2\geq [z^t_{\theta}]^2]\leq e^{-t^{\theta}}.
\end{equation}
Then we obtain from $\|\lambda^{t+1}\|\leq \|\lambda^t\|+\sigma \gamma_1$ and
$\|\lambda^{t+1}\|\leq (1+t)\sigma \gamma_1$ that
$$
\begin{array}{ll}
\mathbb E[\|\lambda^{t+1}\|^2\,|\, \xi_{[t-1]}]&=\mathbb E[\|\lambda^{t+1}\|^2\,|\, \xi_{[t-1]},\|\lambda^t\|<z^t_{\theta}]+\mathbb E[\|\lambda^{t+1}\|^2\,|\, \xi_{[t-1]},
\|\lambda^t\|\geq z^t_{\theta}]\\[10pt]
&\leq [z^t_{\theta}+\gamma_1\sigma]^2+(t+1)^2\sigma^2 \gamma_1^2 e^{-t^{\theta}},
\end{array}
$$
namely (\ref{eq:pda}) holds.

Now sum over $t=1$ to $T$. Using $z^t_\theta = \psi + 16\frac{\gamma_1^2}{\varepsilon_0}\sigma s t^\theta$,
\[
(z^t_\theta)^2 = \psi^2 + 32\psi\frac{\gamma_1^2}{\varepsilon_0}\sigma s t^\theta + 256\frac{\gamma_1^4}{\varepsilon_0^2}\sigma^2 s^2 t^{2\theta}.
\]
Summing $t^\theta$ and $t^{2\theta}$ gives the bounds $\frac{T^{\theta+1}}{\theta+1}$ and $\frac{T^{2\theta+1}}{2\theta+1}$, respectively. Also,
\[
\sum_{t=1}^T z^t_\theta \le T\psi + 16\frac{\gamma_1^2}{\varepsilon_0}\sigma s \frac{T^{\theta+1}}{\theta+1}.
\]
For the exponential tail, note that
\[
\sum_{t=1}^T (t+1)^2 e^{-t^\theta} \le \int_1^\infty (t+1)^2 e^{-t^\theta} dt.
\]
Substituting $u = t^\theta$, $dt = \frac{1}{\theta} u^{1/\theta - 1} du$, we obtain
\[
\int_1^\infty (t^2 + 2t + 1) e^{-t^\theta} dt
= \frac{1}{\theta} \left[ \Gamma\Big(\frac{3}{\theta},1\Big) + 2\Gamma\Big(\frac{2}{\theta},1\Big) + \Gamma\Big(\frac{1}{\theta},1\Big) \right].
\]
Collecting all terms yields the desired bound (\ref{eq:sumpda}).
\hfill $\Box$

In view of Lemma \ref{lem:ebound} and Lemma \ref{lem:pdsa}, we obtain the following result.
\begin{lemma}\label{lem:lamsum}
Let   the assumptions in Lemma \ref{lem:ebound} be satisfied.
For any $\theta>0$, one has for any integer $T>1$,
\begin{equation}\label{eqlamd1}
\frac{1}{T}\mathbb E \displaystyle \displaystyle \sum_{t=1}^T \|\lambda^t\| \leq \psi (\sigma, \alpha,\tau, \tau)
\end{equation}
and
\begin{equation}\label{eqlamd2}
\displaystyle \frac{1}{T}\displaystyle \mathbb E \sum_{t=1}^T \big[ \|\lambda^{t+1}\|^2 \big]
\leq \pi (\sigma,\alpha,\tau, \theta)
\end{equation}
with
\begin{equation}\label{eqlcon}
\begin{array}{ll}
\pi (\sigma,\alpha,\tau, \theta)&= \psi(\sigma,\alpha,\tau,\tau)^2 + 32\psi(\sigma,\alpha,\tau,\tau)\displaystyle \frac{\gamma_1^2}{\varepsilon_0}\sigma \tau \frac{T^{\theta}}{\theta+1}
+ 256\displaystyle \frac{\gamma_1^4}{\varepsilon_0^2}\sigma^2 \tau^2 \displaystyle \frac{T^{2\theta}}{2\theta+1} \\[12pt]
&\quad + 2\sigma\nu_g \left[ \psi(\sigma,\alpha,\tau,\tau) + 16\displaystyle \frac{\gamma_1^2}{\varepsilon_0}\sigma \tau  \frac{T^{\theta}}{\theta+1} \right]
+ (\sigma\nu_g)^2 \\[12pt]
&\quad + \displaystyle \frac{(\sigma\nu_g)^2}{T \theta} \left[ \Gamma\Big(\displaystyle \frac{3}{\theta},1\Big) + 2\Gamma\Big(\displaystyle \frac{2}{\theta},1\Big) + \Gamma\Big(\displaystyle \frac{1}{\theta},1\Big) \right].
\end{array}
\end{equation}
If we choose
$$
\alpha=\beta T^{1/4},\sigma=T^{-3/4},\tau=T^{1/2}, s=T^{1/2}, \theta=1/2,
$$
then
\begin{equation}\label{eqSUMl}
\frac{1}{T}\mathbb E \displaystyle \displaystyle \sum_{k=1}^T \|\lambda^t\| \leq \gamma_3
\end{equation}
and
\begin{equation}\label{eqSUM2}
\frac{1}{T}\mathbb E \displaystyle \displaystyle \sum_{k=1}^T \|\lambda^t\|^2
\leq \pi (\beta T^{1/4},T^{-3/4}, T^{1/2}) \leq \gamma_4
\end{equation}
where
\begin{equation}\label{eq:gam4}
\gamma_4= \gamma_3^2 + 32\gamma_3\displaystyle \frac{\gamma_1^2}{\varepsilon_0}
+ 256\displaystyle \frac{\gamma_1^4}{\varepsilon_0^2}  + 2\sigma\nu_g \left[ \gamma_3 + 16\displaystyle \frac{\gamma_1^2}{\varepsilon_0} \right]
+ \nu_g^2 + 2\nu_g^2 \left[ \Gamma\Big(6,1\Big) + 2\Gamma\Big(4,1\Big) + \Gamma\Big(2,1\Big) \right].
\end{equation}
\end{lemma}

\section{Sample Complexities of PMQSopt}\label{Sec3}
\setcounter{equation}{0}
 Since Problem (\ref{eq:1}) is nonconvex, we can only discuss how the sequence generated by PMQSopt satisfies the Karush-Kuhn-Tucker (KKT) conditions of Problem (\ref{eq:1}). Define the Lagrange function of Problem (\ref{eq:1})
$$
L(x,\lambda)=f(x)+\lambda^Tg(x).
$$
And for given sample $\xi\in \Xi$, define the Lagrange function associated with $\xi$
$$
{\cal L} (x,\lambda;\xi)=F(x,\xi)+\langle \lambda , G(x,\xi)\rangle,
$$
then $\mathbb E \left[{\cal  L} (x,\lambda;\xi)\right]=f(x)+\langle \lambda,g(x)\rangle=L(x,\lambda)$.
We say $(x^*,\lambda^*)\in {\cal O}_0 \times \mathbb R^p$ satisfies the KKT conditions of Problem (\ref{eq:1}) if
\begin{equation}\label{eq:KKToriginalP}
\left\{
\begin{array}{l}
0\in \nabla_x L(x^*,\lambda^*)+N_{X_0}(x^*),\\[6pt]
0\geq g(x^*) \bot \lambda^* \geq 0,
\end{array}
\right.
\end{equation}
where $N_{X_0}(x^*)$ is the normal cone of $X_0$ at $x^* \in X_0$ in the sense of convex analysis, see \cite{Rock70}. The conditions (\ref{eq:KKToriginalP}) are equivalent to
\begin{equation}\label{eq:KKTref}
\left\{
\begin{array}{l}
\|R_{\alpha}(x^*,\lambda^*)\|=0,\\[6pt]
0\geq g(x^*) \bot \lambda^* \geq 0,
\end{array}
\right.
\end{equation}
where
$$
R_{\alpha}(x,\lambda)=\alpha [x-\Pi_{X_0}(x-\alpha^{-1}\nabla_x L(x,\lambda))].
$$
This leads us to define $\varepsilon$-approximate KKT point.  We say $(x,\lambda)$ is a $\varepsilon$-approximate KKT point if the following conditions hold:
\begin{equation}\label{eq:KKTrefa}
\left\{
\begin{array}{l}
\|R_{\alpha}(x,\lambda)\|\leq \varepsilon, -\langle \lambda, g(x)\rangle  \leq \varepsilon,\\[6pt]
 g(x)\leq \varepsilon \textbf{1}_p, \lambda \geq -\varepsilon \textbf{1}_p.
\end{array}
\right.
\end{equation}
In this section, we develop oracle complexities of PMQSopt for finding an $\varepsilon$-approximate KKT point of Problem (\ref{eq:1}).

Define
\begin{equation}\label{eq:bk}
\beta_k( \sigma)=L_0+\left(\displaystyle \sum_{j=1}^pL_j\gamma_2\right)k \sigma.
\end{equation}
\begin{lemma}\label{eq:bk}
Let Assumptions (A1)--(A4),(A6), (B2) and (B3) be satisfied. Then $x \rightarrow L(x,\lambda^k)$ is $\beta_k(\sigma)$-weakly convex.
\end{lemma}
{\bf Proof}. Since (A6) holds,  $\lambda\geq 0$,
 $$ x \rightarrow L(x,\lambda^k)=f(x)+\displaystyle \sum_{j=1}^p \lambda^k_jg_j(x)$$
 is $L_0+\displaystyle \sum_{j=1}^p \lambda^k_jL_j$-weakly convex. In view of (\ref{eq:lambdainorm}) and $\lambda^1=0$, we have $\lambda^k_j \leq \gamma_2 k \sigma$. This observation yields the result. \hfill $\Box$
 \begin{lemma} \label{lemaugP}
Let Assumptions (A2)-(A4) and (B3) be satisfied. For $\lambda \ge 0$ and $\sigma > 0$, one has that $x \rightarrow \mathcal{L}_{\sigma}^{k}(x, \lambda)$ is convex over $X_0$ if $\Sigma_{0}^{k} + \sum_{j=1}^{p} (\lambda_{j} + \sigma \gamma_2) \Sigma_{j}^{k}$ is positive semidefinite, where $\gamma_2 = \nu_g + \kappa_g D_0 + \frac{1}{2} \kappa_\Sigma D_0^2$.
\end{lemma}
{\bf Proof}.
Noting that
\begin{equation*}
    \mathcal{L}_{\sigma}^{k}(x, \lambda) = \min_{y \le 0} \hat{\phi}(x, y, \lambda) = q_{0}^{k}(x) + \langle \lambda, q^{k}(x) - y \rangle + \frac{\sigma}{2} \| q^{k}(x) - y \|^{2},
\end{equation*}
we evaluate the second-order derivatives of $\hat{\phi}(x, y, \lambda)$ with respect to $x$ and $y$. Since $\nabla^{2}q_{j}^{k}(x) = \Sigma_{j}^{k}$ for $j=0,1,\ldots,p$, it is easy to obtain
\begin{equation*}
    \nabla_{x,y}^{2}\hat{\phi}(x, y, \lambda) =
    \begin{bmatrix}
    \Sigma_{0}^{k} + \sum_{j=1}^{p} (\lambda_{j} + \sigma(q_{j}^{k}(x) - y_{j})) \Sigma_{j}^{k} + \sigma \mathcal{J}q^{k}(x)^{T} \mathcal{J}q^{k}(x) & -\sigma \mathcal{J}q^{k}(x)^{T} \\
    -\sigma \mathcal{J}q^{k}(x) & \sigma I
    \end{bmatrix}.
\end{equation*}
The Schur complement of $\sigma I$ in $\nabla_{x,y}^{2}\hat{\phi}(x, y, \lambda)$ is
\begin{equation*}
    \nabla_{xx}^{2}\hat{\phi}(x,y,\lambda) - (-\sigma\mathcal{J}q^{k}(x)^{T})(\sigma I)^{-1}(-\sigma\mathcal{J}q^{k}(x)) = \Sigma_{0}^{k} + \sum_{j=1}^{p} \left( \lambda_{j} + \sigma(q_{j}^{k}(x) - y_{j}) \right) \Sigma_{j}^{k}.
\end{equation*}
For the profiled function $\mathcal{L}_{\sigma}^{k}(x, \lambda)$ to be convex, we evaluate this Schur complement at the optimal $y$ that achieves the minimum. For a given $x$, the minimizer $y^{*} \le 0$ is given component-wise by
\begin{equation*}
    y_{j}^{*} = \min\left(0, q_{j}^{k}(x) + \frac{\lambda_{j}}{\sigma}\right).
\end{equation*}
Substituting $y_{j}^{*}$ into the coefficient yields
\begin{equation*}
    \lambda_{j} + \sigma(q_{j}^{k}(x) - y_{j}^{*}) = \max\left(0, \lambda_{j} + \sigma q_{j}^{k}(x)\right) = [\lambda_{j} + \sigma q_{j}^{k}(x)]_{+}.
\end{equation*}
Thus, the exact condition for convexity at $x$ is that $\Sigma_{0}^{k} + \sum_{j=1}^{p} [\lambda_{j} + \sigma q_{j}^{k}(x)]_{+} \Sigma_{j}^{k}$ is positive semidefinite.

To ensure global convexity over $X_0$, we introduce the uniform upper bound for $q_j^k(x)$. From the definition of $q_j^k(x)$ and Assumptions (A2)-(A4) and (B3), we have for any $x \in X_0$:
\begin{align*}
    q_{j}^{k}(x) &= G_{j}(x^{k}, \xi_{k}) + \langle \nabla_{x} G_{j}(x^{k}, \xi_{k}), x - x^{k} \rangle + \frac{1}{2} \langle \Sigma_{j}^{k} (x - x^{k}), x - x^{k} \rangle \\
    &\le \nu_g + \kappa_g D_0 + \frac{1}{2} \kappa_\Sigma D_0^2 = \gamma_2.
\end{align*}
Since $\lambda_j \ge 0$ and $\sigma > 0$, we have $0 \le [\lambda_{j} + \sigma q_{j}^{k}(x)]_{+} \le \lambda_{j} + \sigma \gamma_2$.

By Assumption (B3), $\Sigma_{j}^{k}$ is negative semidefinite ($\Sigma_{j}^{k} \preceq 0$). Multiplying a negative semidefinite matrix by a larger non-negative scalar yields a "more negative" matrix in the Loewner order, which implies:
\begin{equation*}
    [\lambda_{j} + \sigma q_{j}^{k}(x)]_{+} \Sigma_{j}^{k} \succeq (\lambda_{j} + \sigma \gamma_2) \Sigma_{j}^{k}.
\end{equation*}
Summing over $j$ and adding $\Sigma_{0}^{k}$, we obtain:
\begin{equation*}
    \Sigma_{0}^{k} + \sum_{j=1}^{p} [\lambda_{j} + \sigma q_{j}^{k}(x)]_{+} \Sigma_{j}^{k} \succeq \Sigma_{0}^{k} + \sum_{j=1}^{p} (\lambda_{j} + \sigma \gamma_2) \Sigma_{j}^{k}.
\end{equation*}
Therefore, if $\Sigma_{0}^{k} + \sum_{j=1}^{p} (\lambda_{j} + \sigma \gamma_2) \Sigma_{j}^{k} \succeq 0$, the Schur complement is positive semidefinite for all $x \in X_0$, making $x \rightarrow \mathcal{L}_{\sigma}^{k}(x, \lambda)$ convex over $X_0$.
\hfill $\Box$
 \begin{proposition}\label{corPosiSig}
 If (A1)-(A4), (A6) hold,
  $-(L_i+1) I \prec \Sigma^k_i \prec -L_i I$ for $i=1,\ldots p$, $\kappa_{\Sigma}=\max\{L_1,\ldots, L_p\}+1$,
  and
 \begin{equation}\label{eq:Sig0}
 \Sigma^k_0 =-\displaystyle \sum_{i=1}^p\lambda^k_i \Sigma^k_i+\tau I
 \end{equation}
 for  $\tau >p\kappa_{\Sigma}\gamma_2\sigma$ with $\gamma_2 = \nu_g + \kappa_g D_0 + \frac{1}{2} \kappa_\Sigma D_0^2$, then  Assumptions (B1) -- (B4) hold.
 \end{proposition}

In the following, we will discuss properties of the sequence generated by PMQSopt, which are related to the following three measures:
$$
\displaystyle \frac{1}{T}\sum_{k=1}^T \mathbb E R_{\alpha}(x^k,\lambda^k),\,\quad
\displaystyle \frac{1}{T}\sum_{k=1}^T \mathbb E g(x^k),\,\quad -\displaystyle \frac{1}{T}\mathbb E\left( \displaystyle \sum_{t=1}^T\langle \lambda^t, g(x^t) \rangle\right)
$$
or
$$
\mathbb E_{R,\xi_{[T]}} R_{\alpha}(x^{R},\lambda^{R}),\quad \,
\mathbb E_{R, \xi_{[T]}} g(x^{R}),\quad \,
\mathbb E_{R, \xi_{[T]}} \langle \lambda^{R}, g(x^{R}).
$$
For a given $\lambda^t$, define
\begin{equation}\label{defphit}
\phi^t(x)=L(x,\lambda^t)+\delta_{X_0}(x).
\end{equation}
It follows from \cite[Theorem 4.5]{Drusvyatskiy2018} that
\begin{equation}\label{eq:re-ms}
\displaystyle \frac{1}{4}\|\nabla \phi^t_{1/\alpha}(x)\|\leq \|R_{\alpha/2}(x)\|
\leq \displaystyle \frac{3}{2}\left(1+\displaystyle \frac{1}{\sqrt{2}}\right)\|\nabla \phi^t_{1/\alpha}(x)\|, \, \forall x \in {\cal O}_0.
\end{equation}
Thus, instead of using $\|R_{\alpha/2}(x^t, \lambda^t)\|$, we may use  $\|\nabla \phi^t_{1/\alpha}(x^t)\|$ to measure the discrepancy of $-\nabla L(x^k,\lambda^k)$ from $N_{X_0}(x^k)$.

\begin{theorem}\label{th:L-est}
If (A1)-(A4), (A6) hold,
  $-(L_i+1) I \prec \Sigma^k_i \prec -L_i I$ for $i=1,\ldots p$, $\kappa_{\Sigma}=\max\{L_1,\ldots, L_p\}+1$,
  and
 \begin{equation}\label{eq:Sig0s}
 \Sigma^k_0 =-\displaystyle \sum_{j=1}^p \lambda^k_j\Sigma^k_j+\tau I
 \end{equation}
 for some positive number $\tau >p\kappa_{\Sigma}\gamma_2\sigma$,  where $\gamma_2 = \nu_g + \kappa_g D_0 + \frac{1}{2} \kappa_\Sigma D_0^2$.
 Suppose
\begin{equation}\label{eqal1}
\alpha>  2\beta_k(\sigma),
\end{equation}
Then
\begin{equation}\label{eqfinalesti}
\begin{array}{ll}
 \|\nabla\phi^t_{1/\alpha}(x^t)\|^2
& \leq {4(\alpha+\tau)} \left[ \phi^t_{1/\alpha}(x^t)-\phi^{t+1}_{1/\alpha}(x^{t+1})
\right]+\displaystyle {4(\alpha+\tau)}\nu_g\gamma_1\sigma\\[8pt]
&\quad  + 4\alpha \sigma D_0\gamma_2p[\kappa_g+\kappa_{\Sigma}D_0]+\displaystyle \frac{2\alpha }{ \alpha+\tau}[\kappa_f+\sqrt{p}\kappa_g\|\lambda^t\|+\sigma \gamma_2p(\kappa_g+\kappa_{\Sigma}D_0)]^2.
\end{array}
\end{equation}
\end{theorem}
{\bf Proof}. From the necessary optimality condition for (\ref{xna}), we have
$$
0 \in \nabla q^t_0(x^{t+1})+{\cal J}q^t(x^{t+1})^T[\lambda^t+\sigma q^t(x^{t+1})]_++\alpha (x^{t+1}-x^t)+N_{X_0}(x^{t+1}).
$$
In terms of the definition of $\lambda^{t+1}$, we have
$$
0 \in \nabla q^t_0(x^{t+1})+{\cal J}q^t(x^{t+1})^T\lambda^{t+1}+\alpha (x^{t+1}-x^t)+N_{X_0}(x^{t+1})
$$
or
\begin{equation}\label{hel1}
0 \in \nabla_x F(x^t,\xi_t)+\displaystyle \sum_{i=1}^p
\lambda^{t}_i \nabla_xG_i(x^t,\xi_t)+\Delta_t
+N_{X_0}(x^{t+1}).
\end{equation}
where
$$
\begin{array}{ll}
\Delta_t & =
\left[\Sigma^t_0+\alpha I+\displaystyle \sum_{j=1}^p \lambda^t_j \Sigma^t_j\right](x^{t+1}-x^t)+\displaystyle \sum_{i=1}^p
[\lambda^{t+1}_i-\lambda^t_i] [\nabla_xG_i(x^t,\xi_t)+\Sigma^t_i(x^{t+1}-x^t)]\\[6pt]
&=(\alpha+\tau)(x^{t+1}-x^t)+\widehat \Delta_t
\end{array}
$$
and
$$
\widehat \Delta_t=\displaystyle \sum_{i=1}^p
[\lambda^{t+1}_i-\lambda^t_i] [\nabla_xG_i(x^t,\xi_t)+\Sigma^t_i(x^{t+1}-x^t)].
$$
Then we may rewrite (\ref{hel1}) as
\begin{equation}\label{hel1a}
-\nabla_x {\cal L}(x^t,\lambda^t;\xi_t)-(\alpha+\tau)(x^{t+1}-x^t)-\widehat \Delta_t
\in N_{X_0}(x^{t+1}).
\end{equation}
The inclusion (\ref{hel1a}) is equivalent to
\begin{equation}\label{hel1b}
\begin{array}{ll}
x^{t+1}& =\Pi_{X_0}\left(x^{t+1}+(\alpha+\tau)^{-1}\left[-\nabla_x {\cal L}(x^t,\lambda^t;\xi_t)-(\alpha+\tau)(x^{t+1}-x^t)-\widehat \Delta_t\right]\right)\\[6pt]
& =\Pi_{X_0}\left(x^t-(\alpha+\tau)^{-1} (\nabla_x{\cal L}(x^t,\lambda^t;\xi_t)+\widehat \Delta_t)\right).
\end{array}
\end{equation}
For $\hat x^t={\rm prox}_{\phi^t/\alpha}(x^t)$, then
\begin{equation}\label{hel1c}
\begin{array}{ll}
 \left[\phi^t_{1/\alpha}(x^{t+1})  \right] &
= \displaystyle \inf_z \left\{ \phi^t(z)+\displaystyle
\frac{\alpha}{2}\|z-x^{t+1}\|^2  \right\}\\[6pt]
& \leq \left\{\phi^t_{1/\alpha}(\hat x^t)+\displaystyle
\frac{\alpha}{2}\|\hat x^t-x^{t+1}\|^2  \right\}\\[6pt]
&= \phi^t(\hat x^t)+\displaystyle
\frac{\alpha}{2} \|\Pi_{X_0}\left(x^t-(\alpha+\tau)^{-1} (\nabla_x{\cal L}(x^t,\lambda^t;\xi_t)+\widehat \Delta_t)\right)-\Pi_{X_0}(\hat x^t)\|^2\\[6pt]
& \leq \phi^t(\hat x^t)+\displaystyle
\frac{\alpha}{2} \|x^t-(\alpha+\tau)^{-1} (\nabla_x{\cal L}(x^t,\lambda^t;\xi_t)+\widehat \Delta_t)-\hat x^t\|^2\\[6pt]
&= \phi^t(\hat x^t)+\displaystyle
\frac{\alpha}{2}\|x^t-\hat x^t\|^2 +(\alpha+\tau)^{-1} \alpha  \langle \hat x^t-x^t,\nabla_x{\cal L}(x^t,\lambda^t;\xi_t)+\widehat \Delta_t\rangle\\[6pt]
& \quad \quad + \displaystyle \frac{\alpha (\alpha+\tau)^{-2}}{2}\|\nabla_x{\cal L}(x^t,\lambda^t;\xi_t)+\widehat \Delta_t\|^2\\[6pt]
&= \phi^t_{1/\alpha}(x^t) +(\alpha+\tau)^{-1} \alpha  \langle \hat x^t-x^t,\nabla_x L(x^t,\lambda^t)\rangle\\[6pt]
& \quad \quad +(\alpha+\tau)^{-1} \alpha \langle \hat x^t-x^t, \widehat \Delta_t\rangle+ \displaystyle \frac{\alpha (\alpha+\tau)^{-2}}{2} \|\nabla_x{\cal L}(x^t,\lambda^t;\xi_t)+\widehat \Delta_t\|^2\\[6pt]
&\leq \phi^t_{1/\alpha}(x^t) +(\alpha+\tau)^{-1} \alpha  \left[L(\hat x^t,\lambda^t)-L(x^t,\lambda^t)+\displaystyle
\frac{\beta_t(\sigma)}{2}\|\hat x^t-x^t\|^2 \right]\\[6pt]
& \quad \quad +(\alpha+\tau)^{-1} \alpha \langle \hat x^t-x^t, \widehat \Delta_t\rangle+ \displaystyle \frac{\alpha (\tau+\alpha)^{-2}}{2} \|\nabla_x{\cal L}(x^t,\lambda^t;\xi_t)+\widehat \Delta_t\|^2
\end{array}
\end{equation}
Since $ x \rightarrow L(x,\lambda^t)+\displaystyle \frac{\alpha}{2}\|x-x^t\|^2$ is $\alpha-\beta_k(\sigma)$-strongly convex, we have
$$
\begin{array}{l}
L(x^t,\lambda^t)-L(\hat x^t,\lambda^t)-\displaystyle
\frac{\beta_t(\sigma)}{2}\|\hat x^t-x^t\|^2 \\[6pt]
=\left(L(x^t,\lambda^t)+\displaystyle \frac{\alpha}{2}\|x^t-x^t\|^2\right)-\left(L(\hat x^t,\lambda^t)+\displaystyle \frac{\alpha}{2}\|\hat x^t-x^t\|^2\right)+\displaystyle
\frac{\alpha-\beta_t(\sigma)}{2}\|\hat x^t-x^t\|^2 \\[6pt]
\geq \displaystyle
\frac{\alpha-\beta_t(\sigma)}{2}\|\hat x^t-x^t\|^2 =\displaystyle \frac{\alpha-\beta_t(\sigma)}{2\alpha^2}\|\nabla \phi^t_{1/\alpha}(x^t)\|^2.
\end{array}
$$
Substituting the above inequality to (\ref{hel1c}), we obtain
\begin{equation}\label{hel1d}
\begin{array}{ll}
 \left[\phi^t_{1/\alpha}(x^{t+1})  \right] \leq &\phi^t_{1/\alpha}(x^t)
-\displaystyle \frac{[\alpha-\beta_t(\sigma)](\alpha+\tau)^{-1}}{2\alpha}\|\nabla \phi^t_{1/\alpha}(x^t)\|^2
 \\[6pt]
& \quad \quad +(\alpha+\tau)^{-1} \alpha \langle \hat x^t-x^t, \widehat \Delta_t\rangle+ \displaystyle \frac{\alpha (\alpha+\tau)^{-2}}{2} \|\nabla_x{\cal L}(x^t,\lambda^t;\xi_t)+\widehat \Delta_t\|^2
\end{array}
\end{equation}
Since $\alpha>  2\beta_k(\sigma)$, (\ref{hel1d}) implies
\begin{equation}\label{eqfinalestia}
\begin{array}{ll}
 \|\nabla\phi^t_{1/\alpha}(x^t)\|^2
& \leq \displaystyle{4(\alpha+\tau)} \left[ \phi^t_{1/\alpha}(x^t)-\phi^{t}_{1/\alpha}(x^{t+1})
\right]\\[8pt]
&\quad  + 4\alpha \langle \hat x^t-x^t, \widehat \Delta_t\rangle+ \displaystyle 2\alpha (\alpha+\tau)^{-1} \|\nabla_x{\cal L}(x^t,\lambda^t;\xi_t)+\widehat \Delta_t\|^2\\[8pt]
&\leq \displaystyle {4(\alpha+\tau)} \left[ \phi^t_{1/\alpha}(x^t)-\phi^{t}_{1/\alpha}(x^{t+1})
\right]\\[8pt]
&\quad  + 4\alpha \sigma D_0\gamma_2p[\kappa_g+\kappa_{\Sigma}D_0]+2\alpha (\alpha+\tau)^{-1}[\kappa_f+\sqrt{p}\kappa_g\|\lambda^t\|+\sigma \gamma_2p(\kappa_g+\kappa_{\Sigma}D_0)]^2.
\end{array}
\end{equation}
Noting for $w^{t+1}={\rm prox}_{\alpha^{-1}\phi_t}(x^{t+1})$, we have
$$
\begin{array}{l}
\phi^{t+1}_{1/\alpha}(x^{t+1})-\phi^{t}_{1/\alpha}(x^{t+1})\\[6pt]
\leq L(w^{t+1},\lambda^{t+1})+\displaystyle \frac{\alpha}{2}
\|w^{t+1}-x^{t+1}\|^2-\left(L(w^{t+1},\lambda^{t})+\displaystyle \frac{\alpha}{2}
\|w^{t+1}-x^{t+1}\|^2  \right)\\[10pt]
=(\lambda^{t+1}-\lambda^t)^Tg(w^{t+1})
\leq \nu_g\|\lambda^{t+1}-\lambda^t\|\leq \nu_g\gamma_1\sigma.
\end{array}
$$
Combing this with (\ref{eqfinalestia}) yields
(\ref{eqfinalesti}). \hfill $\Box$


The following result is crucial for estimating the constraint violation.
\begin{proposition}\label{prop:cregret}
Let $(x^t,\lambda^t)$ be generated by PMQSopt and Assumptions (A2), (A4) and  (B3) be satisfied.  Then for  $i=1,\ldots,p$,
\begin{equation}\label{eq:ccomineq0}
\begin{array}{ll}
\displaystyle \sum_{t=1}^T G_i(x^t,\xi_t)
&\leq
\displaystyle \frac{1}{\sigma} \lambda^{T+1}_i+\left[\kappa_g+\displaystyle \frac{\kappa_{\Sigma}D_0}{2}\right]\sum_{t=1}^T\|x^{t+1}-x^t\|
\end{array}
\end{equation}
and
\begin{equation}\label{eq:ccomineqaa}
\mathbb E \displaystyle \sum_{t=1}^TG_i(x^t,\xi_t)\leq \displaystyle \frac{1}{\sigma}\mathbb E \lambda^{T+1}_i+\left[\kappa_g+\displaystyle \frac{\kappa_{\Sigma}D_0}{2}\right]\sum_{t=1}^T\mathbb E\|x^{t+1}-x^t\|.
\end{equation}
\end{proposition}
{\bf Proof}. From the definition $\lambda^{t+1}_i=[\lambda^t_i+\sigma q^t_i(x^{t+1})]_+$, we have  that
$$
\begin{array}{ll}
\lambda^{t+1}_i
&\geq \lambda^t_i+\sigma\left( G_i(x^t,\xi_t)+\langle \nabla_x G_i(x^t,\xi_t), x^{t+1}-x^t\rangle +\displaystyle \frac{1}{2} \left \langle\Sigma^t_i (x^{t+1}-x^t),x^{t+1}-x^t \right \rangle \right)\\[10pt]
& \geq \lambda^t_i+\sigma\left(G_i(x^t,\xi_t)- \|\nabla_x G_i(x^t,\xi_t)\|\|x^{t+1}-x^t\|-\displaystyle \frac{1}{2} \|\Sigma^t_i\|\| x^{t+1}-x^t\|^2\right),
\end{array}
$$
which, from Assumptions (B3) and (A4), implies that
$$
\begin{array}{ll}
\displaystyle \sum_{t=1}^T G_i(x^t,\xi_t)&\leq  \displaystyle \frac{1}{\sigma} \lambda^{T+1}_i+ \sum_{t=1}^T\|\nabla_x G_i(x^t,\xi_t)\|\|x^{t+1}-x^t\|+ \displaystyle \frac{1}{2} \sum_{t=1}^T\|\Sigma^t_i\|\| x^{t+1}-x^t\|^2\\[10pt]
&\leq
\displaystyle \frac{1}{\sigma} \lambda^{T+1}_i+ \kappa_g\sum_{t=1}^T\|x^{t+1}-x^t\|+\displaystyle \frac{\kappa_{\Sigma}D_0}{2}\sum_{t=1}^T\|x^{t+1}-x^t\|,
\end{array}
$$
which is just (\ref{eq:ccomineq0}). From this we obtain
(\ref{eq:ccomineqaa}). \hfill $\Box$\\
The following proposition is about the total complementarity violation.
\begin{proposition}\label{prop:compregret}
Let $(x^t,\lambda^t)$ be generated by PMQSopt and Assumptions (A3), (A4) and (B1), (B3) and (B4)be satisfied. Then
\begin{equation}\label{eq:comineq0}
\begin{array}{ll}
-\displaystyle \sum_{t=1}^T\langle \lambda^t, G(x^t,\xi_t)\rangle
&\leq
\displaystyle \frac{1}{2\sigma} [\|\lambda^1\|^2-\|\lambda^{T+1}\|^2]+\displaystyle \frac{\sigma}{2}\sum_{t=1}^T\|G(x^t,\xi_t)\|^2+
\displaystyle \frac{1}{2\alpha}\sum_{t=1}^T\|\nabla_x F(x^t,\xi^t)\|^2
\end{array}
\end{equation}
and
\begin{equation}\label{eq:comcineqb}
-\mathbb E\left( \displaystyle \sum_{t=1}^T\langle \lambda^t, g(x^t) \rangle\right)\leq \displaystyle \frac{\sigma}{2}\nu_g^2T+\displaystyle \frac{1}{2\alpha}\kappa_f^2T.
\end{equation}
\end{proposition}
{\bf Proof}. It follows from (\ref{eq:opt-x-2}) that
$$
\begin{array}{ll}
\displaystyle\langle \nabla_x F(x^t,\xi_t),x^{t+1}-x^t \rangle + \frac{1}{2\sigma}\|\lambda^{t+1}\|^2
+ \alpha \|x^{t+1}-x^t\|^2+\displaystyle \frac{1}{2}\|x^{t+1}-x^t\|^2_{\Sigma^t_0}\\[10pt]
\leq  \displaystyle\frac{1}{2\sigma}\|[\lambda^t+\sigma G(x^t,\xi_t)]_+\|^2\leq \displaystyle\frac{1}{2\sigma}\|[\lambda^t+\sigma G(x^t,\xi_t)]\|^2
\end{array}
$$
which implies
$$
\begin{array}{l}
-\langle \lambda^t, G(x^t,\xi_t)\rangle  \leq  \displaystyle \frac{1}{2\sigma}[\|\lambda^t\|^2-\|\lambda^{t+1}\|^2] -\displaystyle\langle \nabla_x F(x^t,\xi_t),x^{t+1}-x^t \rangle\\[10pt]
\quad \quad - \alpha \|x^{t+1}-x^t\|^2-\displaystyle \frac{1}{2}\|x^{t+1}-x^t\|^2_{\Sigma^t_0}
+ \displaystyle\frac{\sigma}{2}\| G(x^t,\xi_t)\|^2\\[10pt]
\leq \displaystyle \frac{1}{2\sigma}[\|\lambda^t\|^2-\|\lambda^{t+1}\|^2]+\displaystyle \frac{1}{2\alpha}\|\nabla_x F(x^t,\xi_t)\|^2-\displaystyle \frac{\alpha}{2} \|x^{t+1}-x^t\|^2
 -\displaystyle \frac{1}{2}\|x^{t+1}-x^t\|^2_{\Sigma^t_0}
+ \displaystyle\frac{\sigma}{2}\| G(x^t,\xi_t)\|^2.
\end{array}
$$
Making a sum from $1$ to $T$, we obtain (\ref{eq:comineq0}).

Noting that
$\lambda^1=0$ and
$$
\begin{array}{ll}
-\mathbb E\left(\displaystyle \sum_{t=1}^T\langle \lambda^t, G(x^t,\xi_t)\rangle\right)
& =- \mathbb E\mathbb E\left[\displaystyle \sum_{t=1}^T\langle \lambda^t, G(x^t,\xi_t)\rangle\,|\, \xi_{[t-1]}\right]\\[10pt]
& =-\mathbb E \left(\displaystyle \sum_{t=1}^T\langle \lambda^t, \mathbb E[G(x^t,\xi_t)\,|\, \xi_{[t-1]}]\rangle\right)=-\mathbb E\left( \displaystyle \sum_{t=1}^T\langle \lambda^t, g(x^t) \rangle\right)
\end{array}
$$
one has from (\ref{eq:comineq0}) that
\begin{equation}\label{eq:comineq0a}
\begin{array}{ll}
-\mathbb E \left(\displaystyle \sum_{t=1}^T\langle \lambda^t, g(x^t) \rangle\right)
&\leq \displaystyle \frac{\sigma}{2}\sum_{t=1}^T\mathbb E[\|G(x^t,\xi_t)\|^2]+
\displaystyle \frac{1}{2\alpha}\sum_{t=1}^T\mathbb E[\|\nabla_x F(x^t,\xi^t)\|^2]
\end{array}
\end{equation}
which yields (\ref{eq:comcineqb}) from Assumptions (A3) and (A4).\hfill $\Box$

\begin{proposition}\label{prop:regrets}
Let $(x^t,\lambda^t)$ be generated by PMQSopt. If (A1)-- (A6) hold, and
 \begin{equation}\label{eq:Sig0saa}
  \left\{
 \begin{array}{l}
 -(L_i+1)I \prec \Sigma^k_i \prec -L_i I, i=1,\ldots p,\\[6pt]
 \kappa_{\Sigma}=\max\{L_1,\ldots, L_p\}+1,\\[6pt]
 \Sigma^k_0 =-\displaystyle \sum_{j=1}^p \lambda^k_j\Sigma^k_j+\tau I
 \end{array}
 \right.
 \end{equation}
 for some positive number $\tau >p\kappa_{\Sigma}\gamma_2\sigma$,  where $\gamma_2 = \nu_g + \kappa_g D_0 + \frac{1}{2} \kappa_\Sigma D_0^2$. Assume $\epsilon_0>0$ and  $\kappa_\Sigma>0$ satisfy
  \begin{equation}\label{extrCon2}
  \sqrt{p}\kappa_\Sigma \leq \epsilon_0.
  \end{equation}
 Suppose
\begin{equation}\label{eqal1aa}
\alpha>  \max\{2\beta_k(\sigma),p(\kappa_g+\kappa_{\Sigma}D_0/2)^2\sigma\}.
\end{equation}
 Then, the following assertions hold:
\begin{itemize}
\item[{\rm (i)}]
The average expected rate of Lagrangian gradient violation is
\begin{equation}\label{eq:LagGx}
\begin{array}{l}
\mathbb E \displaystyle \frac{1}{T}
\sum_{t=1}^T \|\nabla \phi^t_{1/\alpha}(x^t)\|^2  \leq
\displaystyle \frac{4(\alpha+\tau)}{T} \left[ f(x^1)-\inf_{z\in X_0} f(z)\right]+\displaystyle
\frac{4\nu_g \psi(\sigma,\alpha,\tau,\tau)(\alpha+\tau)}{T}\\[6pt]
 +{4(\alpha+\tau)} \nu_g\gamma_1\sigma +4\alpha \sigma D_0\gamma_2p[\kappa_g+\kappa_{\Sigma}D_0]
+4\alpha (\alpha+\tau)^{-1}[\kappa_f+\sigma \gamma_2p(\kappa_g+\kappa_{\Sigma}D_0)]^2\\[8pt]
+4p\kappa_g^2\alpha (\alpha+\tau)^{-1} \pi(\sigma,\alpha,\tau,\theta)+8\alpha (\alpha+\tau)^{-1} \sigma \cdot p\kappa_g^2\gamma_1\cdot \psi(\sigma,\alpha,\tau,\tau)\\[8pt]
\quad \quad +4\alpha (\alpha+\tau)^{-1}\sigma^2\gamma_1^2p\kappa_g^2.
\end{array}
\end{equation}
\item[{\rm (ii)}]The average expected rate of constraint violation is
\begin{equation}\label{eq:constrc1x}
\displaystyle \frac{1}{T}\mathbb E\left[\sum_{t=1}^T g_i(x^t)\right]\leq \displaystyle \frac{1}{\sigma T}\psi (\sigma,\alpha,\tau,\tau) +\displaystyle\frac{1}
{\alpha+\tau}\left\{ \rho_1  +\rho_2\nu_g \sigma +\rho_2\psi (\sigma,\alpha,\tau,\tau)\right\},
\end{equation}
where
$$
\begin{array}{l}
\rho_1=[2\kappa_g+\kappa_{\Sigma}D_0]\kappa_f,\,\,
\rho_2=[2\kappa_g+\kappa_{\Sigma}D_0](\kappa_g+\kappa_{\Sigma}D_0/2)\sqrt{p}.
\end{array}
$$
\end{itemize}
\end{proposition}
{\bf Proof}. Summing (\ref{eqfinalesti}) over $t=1,\ldots, T$ gives
\begin{equation}\label{LagGra1}
\begin{array}{l}
 \displaystyle \frac{1}{T}\sum_{t=1}^T\|\nabla\phi^t_{1/\alpha}(x^t)\|^2
 \leq \displaystyle \frac{4(\alpha+\tau)}{T} \left[ \phi^1_{1/\alpha}(x^1)-\phi^{T+1}_{1/\alpha}(x^{T+1})
\right]+\displaystyle 4(\alpha+\tau) \nu_g\gamma_1\sigma\\[8pt]
+ 4\alpha \sigma D_0\gamma_2p[\kappa_g+\kappa_{\Sigma}D_0]
\quad  +2\alpha (\alpha+\tau)^{-1}\displaystyle \frac{1}{T}\sum_{t=1}^T[\kappa_f+\sqrt{p}\kappa_g\|\lambda^t\|+\sigma \gamma_2p(\kappa_g+\kappa_{\Sigma}D_0)]^2\\[8pt]
\leq \displaystyle \frac{4(\alpha+\tau)}{T} \left[ \phi^1_{1/\alpha}(x^1)-\phi^{T+1}_{1/\alpha}(x^{T+1})
\right]+{4(\alpha+\tau)}\nu_g\gamma_1\sigma+ 4\alpha \sigma D_0\gamma_2p[\kappa_g+\kappa_{\Sigma}D_0]\\[8pt]
\quad  +4\alpha (\alpha+\tau)^{-1}[\kappa_f+\sigma \gamma_2p(\kappa_g+\kappa_{\Sigma}D_0)]^2+4\alpha(\alpha+\tau)^{-1} p\kappa_g^2\displaystyle \frac{1}{T}\sum_{t=1}^T\|\lambda^t\|^2.
\end{array}
\end{equation}
Noting that
$$
\phi^1_{1/\alpha}(x^1)\leq f(x^1)
$$
and
for $a={\rm prox}_{\alpha^{-1}\phi^{T+1}}(x^{T+1})$,
$$
\phi^{T+1}_{1/\alpha}(x^{T+1})=f(a)+\langle \lambda^{T+1}, g(a)\rangle
\geq \inf_{z\in X_0} f(z)-\nu_g \|\lambda^{T+1}\|,
$$
we have
\begin{equation}\label{eq:diffph}
\displaystyle \frac{4(\alpha+\tau)}{T} \left[ \phi^1_{1/\alpha}(x^1)-\phi^{T+1}_{1/\alpha}(x^{T+1})
\right]\leq \displaystyle \frac{4(\alpha+\tau)}{T} \left[ f(x^1)-\inf_{z\in X_0} f(z)\right]+\displaystyle
\frac{4\nu_g \psi(\sigma,\alpha,\tau,\tau)(\alpha+\tau)}{T}.
\end{equation}
From Lemma \ref{lem:lamsum}, we have
\begin{equation}\label{lamdap}
\begin{array}{ll}
\displaystyle \frac{1}{T}\mathbb E\left[\sum_{t=1}^T\|\lambda^t\|^2\right]
& \leq \displaystyle \frac{1}{T}\mathbb E\left[\sum_{t=1}^T[\|\lambda^{t+1}\|+\sigma \gamma_1]^2\right]\\[8pt]
&=\sigma^2\gamma_1^2+2\sigma \gamma_1\displaystyle \frac{1}{T}\mathbb E\left[\sum_{t=1}^T\|\lambda^t\|\right]+\displaystyle \frac{1}{T}\mathbb E\left[\sum_{t=1}^T\|\lambda^{t+1}\|^2\right]\\[8pt]
&\leq \sigma^2\gamma_1^2+2\sigma \gamma_1\psi (\sigma,\alpha,\tau,\tau)+\pi(\sigma,\alpha,\tau,\theta).
\end{array}
\end{equation}
Combing (\ref{eq:diffph}) and (\ref{lamdap}) with (\ref{LagGra1}), we obtain (\ref{eq:LagGx}).

It follows from (\ref{eq:ccomineqaa}), (\ref{eq:diffX}) and $ 2\alpha-p(\kappa_g+\kappa_{\Sigma}D_0/2)^2\sigma\geq \alpha$ that
\[
\begin{array}{ll}
\mathbb E \left[\displaystyle \sum_{t=1}^Tg_i(x^t)\right]&\leq \displaystyle \frac{1}{\sigma}\mathbb E \lambda^{T+1}_i+\displaystyle\frac{2\kappa_g+\kappa_{\Sigma}D_0}
{\alpha+\tau}\sum_{t=1}^T\mathbb E\left(\kappa_f+(\kappa_g+\kappa_{\Sigma}D_0/2)\sqrt{p}
[\|\lambda^t\|+\nu_g\sigma]\right)\\[10pt]
&\leq \displaystyle \frac{1}{\sigma}\psi(\sigma,\alpha,\tau,\tau)+\displaystyle\frac{2\kappa_g+\kappa_{\Sigma}D_0}
{\alpha+\tau}\sum_{t=1}^T\left(\kappa_f+(\kappa_g+\kappa_{\Sigma}D_0/2)\sqrt{p}
[\psi(\sigma,\alpha,\tau,\tau)+\nu_g\sigma]\right),
\end{array}
\]
which proves (\ref{eq:constrc1x}) in (ii). \hfill $\Box$

Basing on Proposition \ref{prop:regrets} and Proposition \ref{prop:compregret}, by choosing  $\sigma=T^{-3/4}$, $\alpha=\beta T^{1/4}$ for some $\beta>0$, and $\tau=s=T^{1/2}$
in PMQSopt, we may derive specific oracle complexity  for  an $\varepsilon$-KKT stationarity.
\begin{theorem}\label{mainth:regrets}
Let $(x^t,\lambda^t)$ be generated by PMQSopt. Let (A1)-- (A6) hold, and
\begin{equation}\label{eqpars}
\alpha=\beta T^{1/4}, \sigma=T^{-3/4}, \tau=T^{1/2}
\end{equation}
where
$$
\beta=2\left[L_0+\displaystyle \sum_{j=1}^p\gamma_2L_j  \right]+1.
$$
Choose
 \begin{equation}\label{eq:Sig0sax}
  \left\{
 \begin{array}{l}
 -(L_i+1)I\prec\Sigma^k_i \prec -L_i I, i=1,\ldots p,\\[6pt]
 \kappa_{\Sigma}=\max\{L_1,\ldots, L_p\}+1,\\[6pt]
 \Sigma^k_0 =-\displaystyle \sum_{j=1}^p \lambda^k_j\Sigma^k_j+T^{1/2}I.
 \end{array}
 \right.
 \end{equation}
 If $T$ satisfies
\begin{equation}\label{eqal1xa}
T> \max\left\{\beta^2, \beta^{-1}p(\kappa_g+\kappa_{\Sigma}D_0/2)^2,(\kappa_{\Sigma}\gamma_2)^4,p\kappa_{\Sigma}\gamma_2\right\}.
\end{equation}
Assume $\epsilon_0>0$ and  $\kappa_\Sigma>0$ satisfy
  \begin{equation}\label{extrCon3}
  \sqrt{p}\kappa_\Sigma \leq \epsilon_0.
  \end{equation}
 Then the following assertions hold:
\begin{itemize}
\item[{\rm (i)}]The average expected rate of Lagrangian gradient violation is
\begin{equation}\label{eq:LagGxx}
\begin{array}{ll}
\displaystyle \frac{1}{T}\mathbb E
\sum_{k=1}^T \|\nabla \phi^t_{1/\alpha}(x^t)\|^2  \leq &
\displaystyle \frac{4(\beta+1)}{T^{1/2}} \left[ f(x^1)-\inf_{z\in X_0} f(z)\right]\\[8pt]
& +\displaystyle
\frac{4}{T^{1/2}}[(\beta+1)\nu_g\gamma_3+
 \beta D_0\gamma_2p[\kappa_g+\kappa_{\Sigma}D_0]\\[6pt]
& +\displaystyle \frac{4}{T^{1/4}}[ (\beta+1)\nu_g\gamma_1+\beta[\kappa_f+ \gamma_2p(\kappa_g+\kappa_{\Sigma}D_0)]^2+\beta p\kappa_g^2 \gamma_4]\\[8pt]
&+\displaystyle \frac{8\beta}{T} p\kappa_g^2\gamma_1 \gamma_3+\displaystyle \frac{4\beta}{T^{7/4}}\gamma_1^2p\kappa_g^2.
\end{array}
\end{equation}
\item[{\rm (ii)}] The average expected rate of  constraint violation to reach $0$ is
\begin{equation}\label{eq:constrc1xx}
\displaystyle \frac{1}{T}\mathbb E\left[\sum_{t=1}^T G_i(x^t,\xi_t)\right]\leq
\displaystyle \frac{\gamma_5}{T^{1/4}},
\end{equation}
where
$$
\gamma_5=\gamma_3 +\displaystyle\frac{1}
{\beta}\left\{[2\kappa_g+\kappa_{\Sigma}D_0]\kappa_f+ (\nu_g +\gamma_3 )[2\kappa_g+\kappa_{\Sigma}D_0](\kappa_g+\kappa_{\Sigma}D_0/2)\sqrt{p}\right\}
=\gamma_3 +\displaystyle\frac{1}
{\beta}[\rho_1+(\nu_g +\gamma_3 )\rho_2].
$$
\item[{\rm (iii)}]The average expected rate of  complementarity violation to reach $0$ is
\begin{equation}\label{eq:complementarity1Axx}
-\displaystyle \frac{1}{T}\mathbb E\left( \displaystyle \sum_{t=1}^T\langle \lambda^t, g(x^t) \rangle\right)\leq \displaystyle \frac{\nu_g^2}{2T^{3/4}}+\displaystyle \frac{\kappa_f^2}{2\beta T^{1/4}}.
\end{equation}
\end{itemize}
\end{theorem}
{\bf Proof}. Since $\alpha=\beta T^{1/4}$, where $\beta=2\left[L_0+\displaystyle \sum_{j=1}^p\gamma_2L_j  \right]+1$, we have
$$
\alpha> 2\left[L_0+\left(\displaystyle \sum_{j=1}^pL_j\gamma_2\right)T\sigma\right]+T^{1/4}
> 2 \beta_{T}(\sigma)\geq 2\beta_k(\sigma)
$$
for any $k=1,\ldots, T$.

Since $T>(\kappa_{\Sigma}\gamma_2)^4$ or equivalently $T^{1/4}>\kappa_{\Sigma}\gamma_2$, $\tau
> \sigma \kappa_{\Sigma}\gamma_2 k$, we have
$$\Sigma^k_0 =-\displaystyle \sum_{j=1}^p \lambda^k_j\Sigma^k_j+\tau I  \succ \sigma \kappa_{\Sigma}\gamma_2 k I.$$

Since $T >\beta^{-1}p(\kappa_g+\kappa_{\Sigma}D_0/2)^2$, we have
$$
\beta T^{1/4}> p(\kappa_g+\kappa_{\Sigma}D_0/2)^2\sigma,
$$
which implies $\alpha >p(\kappa_g+\kappa_{\Sigma}D_0/2)^2\sigma$. Since $T> \beta^2$, we have
$\psi (T^{-3/4},\beta T^{1/4},T^{1/2},T^{1/2})\leq \gamma_3$.

Summarizing the above discussions, we have from (\ref{eq:Sig0sax}) that all conditions in Proposition
\ref{prop:regrets} are satisfied.  Then from Proposition
\ref{prop:regrets}, we obtain (\ref{eq:LagGx}), (\ref{eq:constrc1x}) and
(\ref{eq:comcineqb}) of Proposition \ref{prop:compregret}. Noting that
$$
\psi (T^{-3/4},\beta T^{1/4},T^{1/2}, T^{1/2})\leq \gamma_3, \,
\pi(T^{-3/4},\beta T^{1/4},T^{1/2}, 1/2)\leq \gamma_4,
$$
we obtain the estimates (\ref{eq:LagGxx}), (\ref{eq:constrc1xx}) and (\ref{eq:complementarity1Axx}).
\hfill $\Box$
\begin{corollary}\label{uniform-choice}
Under the conditions of Theorem \ref{mainth:regrets}.
Then the following assertions hold:
\begin{itemize}
\item[{\rm (i)}]The average expected rate of Lagrangian gradient violation is
\begin{equation}\label{eq:LagGxxa}
\begin{array}{ll}
\mathbb E_{R,\xi_{[T]}}
 \|\nabla \phi^{R}_{1/\alpha}(x^{R})\|^2  \leq &
\displaystyle \frac{4(\beta+1)}{T^{1/2}} \left[ f(x^1)-\inf_{z\in X_0} f(z)\right]\\[8pt]
& +\displaystyle
\frac{4}{T^{1/2}}[(\beta+1)\nu_g\gamma_3+
 \beta D_0\gamma_2p[\kappa_g+\kappa_{\Sigma}D_0]\\[6pt]
& +\displaystyle \frac{4}{T^{1/4}}[ (\beta+1)\nu_g\gamma_1+\beta[\kappa_f+ \gamma_2p(\kappa_g+\kappa_{\Sigma}D_0)]^2+\beta p\kappa_g^2 \gamma_4]\\[8pt]
&+\displaystyle \frac{8\beta}{T} p\kappa_g^2\gamma_1 \gamma_3+\displaystyle \frac{4\beta}{T^{7/4}}\gamma_1^2p\kappa_g^2.
\end{array}
\end{equation}
\item[{\rm (ii)}] The average expected rate of  constraint violation to reach $0$ is
\begin{equation}\label{eq:constrc1xxa}
\mathbb E_{R,\xi_{[T]}} g_i(x^{R})\leq
\displaystyle \frac{\gamma_5}{T^{1/4}},
\end{equation} for $i=1,\ldots, p$,
where
$$
\gamma_5
=\gamma_3 +\displaystyle\frac{1}
{\beta}[\rho_1+(\nu_g +\gamma_3 )\rho_2].
$$
\item[{\rm (iii)}]The average expected rate of  complementarity violation to reach $0$ is
\begin{equation}\label{eq:complementarity1Axxa}
-\mathbb E_{R,\xi_{[T]}}
\langle \lambda^{R}, g(x^{R}) \rangle \leq \displaystyle \frac{\nu_g^2}{2T^{3/4}}+\displaystyle \frac{\kappa_f^2}{2\beta T^{1/4}}.
\end{equation}
\end{itemize}
\end{corollary}
\section{High Probability Performance Analysis}\label{Sec4}
\setcounter{equation}{0}
In this section, we shall analyze the high probability performance of PMQSopt under the following ``light-tail"
assumption for  constraint functions:
\begin{description}
\item[(C)] There exists a constant $\rho_c>0$ such that, for all $x\in X_0$,
$$
\mathbb E[\exp{[G_i(x,\xi)-g_i(x)]^2/\rho_c^2}]\leq \exp{\{1\}},\, i=1,\ldots,p.
$$
\end{description}
Assumption C is standard in the analysis of large-deviation property of stochastic algorithms.
Now we will use  part 2 of  Lemma \ref{lem:l7} to establish a high probability constraint violation bound.

Let $s=\lceil T^{1/2}\rceil$, $\mu=\exp{\{-\eta\}}/(T+1)$ and define
\begin{equation}\label{eq:hatphi}
\widehat \phi (T,\eta):=\kappa_0+\kappa_1+(\kappa_1\beta+\kappa_4) T^{-1/4}+\kappa_3 T^{-3/4}
+16\beta\displaystyle \frac{\gamma_1^2}{\varepsilon_0}\left(\eta+\log (T+1)\right)T^{-1/4},
\end{equation}
then we have $\phi (\sigma,\alpha, s,\tau,\mu)\leq \widehat \phi (T,\eta)$ for $\sigma=T^{-3/4}$, $\alpha=\beta
T^{1/4}$ and $\tau=s=T^{1/2}$.
\begin{proposition}\label{th:constr3} Let $\eta >1$.  Let $(x^t,\lambda^t)$ be generated by PMQSopt, and the conditions in Theorem \ref{mainth:regrets} and Assumption C be satisfied. Then
\begin{equation}\label{eq:constrProb}
{\rm Pr} \left[\displaystyle \sum_{t=1}^Tg_i(x^t) \leq \pi_c(T,\eta)\right] \geq 1-\exp{\{-\eta\}}-\exp{\{-\eta^2/3\}},
\end{equation}
where
$$
\pi_c (T,\eta)=\rho_1\beta^{-1}T^{3/4}+\rho_2\beta^{-1}\nu_g+\eta \rho_c\sqrt{T}+  (1+\beta^{-1}\rho_2)\widehat \phi (T,\eta)T^{3/4}.
$$
\end{proposition}
{\bf Proof}. Define $Z(t)=\|\lambda^t\|$ for $\forall t \in \{1,2,\ldots\}$. From Lemma \ref{lem:l5}, $Z(t)$ satisfies the conditions in Lemma \ref{lem:l7} with
$\delta_{\max}=\sigma \gamma_1$ and $\zeta =\displaystyle \frac{\sigma}{4}\epsilon_0$, and $t_0=s$,
 and
$$
  \vartheta (\sigma,\alpha,\tau,s)=\displaystyle \frac{\epsilon_0\sigma s}{4}+\gamma_1\sigma(s-1)+\displaystyle \frac{2(\alpha +\tau)D_0^2}{\epsilon_0s}+\displaystyle \frac{\left(4\kappa_fD_0\right)}{\epsilon_0}+
 \displaystyle \frac{2\sigma \nu_g^2}{\epsilon_0},
 $$
From part 2 of  Lemma \ref{lem:l7}, we obtain  for any constant $0< \mu < 1$, we have
\begin{equation}\label{eq:bd2A}
{\rm Pr}[\|\lambda^k\|\geq \phi (\sigma,\alpha,s,\tau,\mu)
] \leq \mu.
\end{equation}
where
$$
\phi (\sigma,\alpha, s,\tau,\mu)=\kappa_0+\kappa_1\displaystyle \frac{\alpha+\tau}{s}+\kappa_3
 \sigma+\kappa_4 \sigma s+16\displaystyle \frac{\gamma_1^2}{\varepsilon_0}\log \left( \displaystyle \frac{1}{\mu}\right)\sigma s.
$$
If we take $\alpha=\beta T^{1/4}$, $s=\lceil  T^{1/2}\rceil$,$\tau=T^{1/2}$, $\mu=\exp{\{-\eta\}}/(T+1)$, then
$$
\phi (\sigma,\alpha, s,\mu)\leq \kappa_0+\kappa_1+(\kappa_1\beta+\kappa_4) T^{-1/4}+\kappa_3 T^{-3/4}
+16\beta\displaystyle \frac{\gamma_1^2}{\varepsilon_0}\left(\eta+\log (T+1)\right)T^{-1/4}.
$$
Thus (\ref{eq:bd2A}) implies
\begin{equation}\label{eq:p1}
{\rm Pr}\,\left[\|\lambda^t\|\geq \widehat \phi (T,\eta) \right]
\leq \displaystyle \frac{\exp{\{-\eta\}}}{T+1}.
\end{equation}
It follows from (\ref{eq:ccomineq0}) in Proposition \ref{prop:cregret} and (\ref{eq:diffX}) in Lemma \ref{lem:3} and $T > p(\kappa_g+\kappa_{\Sigma}D_0/2)^2$ implying $2\alpha-p(\kappa_g+\kappa_{\Sigma}D_0/2)^2\sigma>0$ that
\begin{equation}\label{eq:ccomineq0A}
\begin{array}{ll}
\displaystyle \sum_{t=1}^T G_i(x^t,\xi_t)
&\leq \displaystyle \frac{1}{\sigma} \lambda^{T+1}_i+\displaystyle\frac{2\kappa_g+\kappa_{\Sigma}D_0}
{\alpha}\sum_{t=1}^T\left(\kappa_f+(\kappa_g+\kappa_{\Sigma}D_0/2)\sqrt{p}
[\|\lambda^t\|+\nu_g\sigma]\right)\\[10pt]
&=\displaystyle \frac{1}{\sigma} \lambda^{T+1}_i+\displaystyle\frac{1}
{\alpha}\left\{ \rho_1 T +\rho_2\nu_g \sigma T+\rho_2\sum_{t=1}^T
\|\lambda^t\|\right\}\\[10pt]
\end{array}
\end{equation}
 Thus, for $\sigma=T^{-3/4}$ and $\alpha=\beta T^{1/4}$,
 we obtain from (\ref{eq:ccomineq0A}) that
\begin{equation}\label{eq:31}
\begin{array}{ll}
\displaystyle \sum_{t=1}^T g_i(x^t)\leq \displaystyle \sum_{t=1}^T [g_i(x^t)-G_i(x^t,\xi_t)]
+T^{3/4} \lambda^{T+1}_i+\displaystyle \frac{\rho_1}{\beta}T^{3/4} +\beta^{-1}\rho_2\nu_g+\displaystyle\frac{\rho_2 }{\beta T^{1/4}}
\sum_{t=1}^T
\|\lambda^t\|
\end{array}
\end{equation}
It follows from From a well-known result (\cite[Lemma 4.1]{Lan2020}), under Assumption C, one has for $i=1,\ldots,p$ that
\begin{equation}\label{eq:32}
{\rm Prob}\left[ \displaystyle \sum_{t=1}^T [g_i(x^t)-G_i(x^t,\xi_t)]\geq \eta \rho_c\sqrt{T}\right]\leq \exp{\{-\eta^2/3\}}.
 \end{equation}
Since (\ref{eq:p1}) holds for every $i=1,\ldots, p$, one has that
\begin{equation}\label{eq:33}
{\rm Prob}\left[ T^{3/4} \lambda^{T+1}_i+\displaystyle\frac{\beta^{-1}\rho_2 }{T^{1/4}}
\sum_{t=1}^T
\|\lambda^t\|\geq (1+\beta^{-1}\rho_2)\widehat \phi (T,\eta)T^{3/4}\right]\leq \exp{\{-\eta\}},
 \end{equation}
 where $\widehat \phi (T,\eta)$ is defined by (\ref{eq:hatphi}).
We next state a simple fact. For two random variables $X, Y$ and three constants $a, b, c $ with $c\in (0,1)$, if ${\rm Pr}[Y \geq  b]\leq c$ and $X \leq Y+a$ almost everywhere, then
$$
{\rm Pr}[X \geq  a+b] \leq {\rm Pr}[Y +a\geq a + b] = {\rm Pr}[Y\geq b] \leq c.
$$
If we take
$$
X=\displaystyle \sum_{t=1}^T g_i(x^t), \,\, Y=\displaystyle \sum_{t=1}^T [g_i(x^t)-G_i(x^t,\xi_t)]
+ T^{3/4} \lambda^{T+1}_i+\displaystyle\frac{\rho_2 }{\beta T^{1/4}}
\sum_{t=1}^T
\|\lambda^t\|
$$
and
$$
a=\rho_1\beta^{-1}T^{3/4}+\rho_2\beta^{-1}\nu_g,\, b=\eta \rho_c\sqrt{T}+  (1+\beta^{-1}\rho_2)\widehat \phi (T,\eta)T^{3/4},
$$
then from (\ref{eq:31}), (\ref{eq:32}) and (\ref{eq:33}) and the above observation, we obtain
$$
{\rm Pr}\left[\displaystyle \sum_{t=1}^T g_i(x^t)\geq a+b\right]
\leq \exp{\{-\eta^2/3\}}+\exp{\{-\eta\}},
$$
which implies (\ref{eq:constrProb}). \hfill $\Box$
 \begin{proposition}\label{th:KKT-res}[High-probability bound for the Moreau envelope gradient]
Let $(x^t,\lambda^t)$ be generated by PMQSopt and let the conditions of Theorem~3.2 hold. For any $\eta>1$,
\begin{equation}\label{eq:KKProb}
\Pr\left[ \frac{1}{T}\sum_{t=1}^{T}\|\nabla\phi_{1/\alpha}^t(x^t)\|^2 \le \pi_{\text{grad}}(T,\eta) \right] \ge 1 - e^{-\eta},
\end{equation}
where
\begin{equation}\label{eq:pidef}
\begin{array}{ll}
\pi_{\text{grad}}(T,\eta) =&\displaystyle \frac{4(\alpha+\tau)}{T}\Bigl[f(x^1)-\inf_{z\in X_0}f(z)\Bigr] +\displaystyle \frac{4\nu_g(\alpha+\tau)}{T}\widehat{\phi}(T,\eta)\\[4pt]
&
+ 4(\alpha+\tau)\nu_g\gamma_1\sigma + 4\alpha\sigma D_0\gamma_2 p[\kappa_g+\kappa_\Sigma D_0]
\\[4pt]
&
+ 4\alpha(\alpha+\tau)^{-1}\bigl[\kappa_f + \sigma\gamma_2 p(\kappa_g+\kappa_\Sigma D_0)\bigr]^2 + 4\alpha(\alpha+\tau)^{-1}p\kappa_g^2\widehat{\phi}(T,\eta)^2.
\end{array}
\end{equation}
\end{proposition}
{\bf Proof}. From  (\ref{LagGra1}) and (\ref{eq:diffph}) we have the following  bound:
\begin{equation}\label{eq:3.27a}
\begin{array}{ll}
\displaystyle\frac{1}{T}\sum_{t=1}^{T}\|\nabla\phi_{1/\alpha}^t(x^t)\|^2 &\le
\displaystyle \frac{4(\alpha+\tau)}{T}\Bigl[f(x^1)-\inf_{z\in X_0}f(z)\Bigr] + \displaystyle \frac{4\nu_g(\alpha+\tau)}{T}\|\lambda^{T+1}\|  \\[8pt]
&\quad + 4(\alpha+\tau)\nu_g\gamma_1\sigma + 4\alpha\sigma D_0\gamma_2 p[\kappa_g+\kappa_\Sigma D_0]  \\[8pt]
&\quad + 4\alpha(\alpha+\tau)^{-1}\bigl[\kappa_f + \sigma\gamma_2 p(\kappa_g+\kappa_\Sigma D_0)\bigr]^2  \\[8pt]
&\quad + 4\alpha(\alpha+\tau)^{-1}p\kappa_g^2\displaystyle \frac{1}{T}\sum_{t=1}^{T}\|\lambda^t\|^2,
\end{array}
\end{equation}
where we have used the fact that $\psi(\sigma,\alpha,\tau,\tau)\le \gamma_3$ for $T\ge \beta^2$ (see (\ref{eq:lamB})).

 From Lemma \ref{lem:ebound}, part (2), with $s = T^{1/2}$ and $\mu = e^{-\eta}/(T+1)$, we have for $t=1,\ldots, T+1$,
\[
\Pr\left[ \|\lambda^t\| \ge \widehat{\phi}(T,\eta) \right] \le \frac{e^{-\eta}}{T+1}.
\]
We introduce the following event
 $$
 B=\left\{\forall t\in \{1,\ldots,T+1\}:\|\lambda^{t}\| \leq \widehat \phi (T,\eta)\right\}.
 $$
  Then  for the complement of $B$, denoted by $B^c$, one has
 $$
 {\rm Prob}\,B^c\leq \exp(-\eta).
 $$
Hence, with probability at least $1 - e^{-\eta}$, it holds that
\[
\|\lambda^{T+1}\| \le \widehat{\phi}(T,\eta) \mbox{ and } \frac{1}{T}\sum_{t=1}^{T}\|\lambda^t\|^2 \le \widehat{\phi}(T,\eta)^2.
\]
Substituting these bounds into (\ref{eq:3.27a}) yields the desired inequality. \hfill $\Box$

 Now we discuss complementarity violation of the sequence $(x^t,\lambda^t)$ b generated by PMQSopt.
\begin{proposition}\label{th:comp-res}
Let $\eta >1$.  Let $(x^t,\lambda^t)$ be generated by PMQSopt, and the conditions in Theorem \ref{mainth:regrets} and Assumption C be satisfied. Then
\begin{equation}\label{eq:Com-ResProb}
{\rm Pr} \left[-\displaystyle \sum_{t=1}^T\langle \lambda^t, g(x^t)\rangle\leq
\pi_{{cm}}(T,\eta)
\right] \geq 1-\exp{\{-\eta\}}-2\exp{\{-\eta^2/3\}},
\end{equation}
where
$$
\pi_{{cm}}(T,\eta)=
p\eta\widehat \phi (T,\eta) \rho_cT^{1/2}+
\displaystyle \frac{\nu_g^2}{2}T^{1/4}+
\displaystyle \frac{\kappa_f^2}{2}\beta^{-1}T^{3/4}
$$
\end{proposition}
{\bf Proof}.
If we take $\sigma=T^{-3/4}$,$\alpha=\beta T^{1/4}$, $s=\lceil  T^{1/2}\rceil$,$\tau=T^{1/2}$, $\mu=\exp{\{-\eta\}}/(T+1)$, then
 (\ref{eq:bd2A}) implies
\begin{equation}\label{eq:p111a}
{\rm Pr}\,\left[\|\lambda^t\|\geq \widehat \phi (T,\eta)\right]
\leq \displaystyle \frac{\exp{\{-\eta\}}}{T+1}.
\end{equation}
In view  of (\ref{eq:comineq0}) of Proposition \ref{prop:compregret}, we have from Assumption (A3) and Assumption (A4) that
\begin{equation}\label{eq:comineq0a}
\begin{array}{ll}
-\displaystyle \sum_{t=1}^T\langle \lambda^t, g(x^t)\rangle
&=-\displaystyle \sum_{t=1}^T\langle \lambda^t,g(x^t)- G(x^t,\xi_t)\rangle
-\displaystyle \sum_{t=1}^T\langle \lambda^t, G(x^t,\xi_t)\rangle\\[10pt]
&\leq\displaystyle \sum_{t=1}^T\langle \lambda^t, G(x^t,\xi_t)-g(x^t)\rangle\\[10pt]
&\quad +\displaystyle \frac{1}{2\sigma} [\|\lambda^1\|^2-\|\lambda^{T+1}\|^2]+\displaystyle \frac{\sigma}{2}\sum_{t=1}^T\|G(x^t,\xi_t)\|^2+
\displaystyle \frac{1}{2\alpha}\sum_{t=1}^T\|\nabla_x F(x^t,\xi^t)\|^2\\[10pt]
&\leq\displaystyle \sum_{t=1}^T\langle \lambda^t, G(x^t,\xi_t)-g(x^t)\rangle+
\displaystyle \frac{\nu_g^2}{2}T^{1/4}+
\displaystyle \frac{\kappa_f^2}{2}\beta^{-1}T^{3/4}.
\end{array}
\end{equation}
We introduce the following event
 $$
 B=\left\{\forall t\in \{1,\ldots,T\}:\|\lambda^{t}\| \leq \widehat \phi (T,\eta)\right\}.
 $$
  Then from  (\ref{eq:p111a}), for the complement of $B$, denoted by $B^c$, one has
 $$
 {\rm Prob}\,B^c\leq \exp(-\eta).
 $$
 From the convexity of $\exp$, one has from Assumption C that
 $$
\mathbb E \exp\left\{\left[\displaystyle \sum_{i=1}^p\left| g_i(x^t)- G_i(x^t,\xi_t) \right| \right]^2/(p^2  \rho_c^2)\right\}
 \leq \mathbb E \displaystyle \sum_{i=1}^p\frac{1}{p}\exp\left\{\left[ g_i(x^t)- G_i(x^t,\xi_t) \right]^2/  \rho_c^2\right\}\leq \exp \{1\}.
 $$

 Thus we have from \cite[Lemma 4.1]{Lan2020} and the above inequality that
 \begin{equation}\label{eq:item4}
 \begin{array}{l}
 {\rm Prob}\left\{\left|\displaystyle \sum_{t=1}^T\langle \lambda^t, G(x^t,\xi_t)-g(x^t)\rangle \right|\geq p\eta\widehat \phi (T,\eta) \rho_cT^{1/2}\right\}\\[10pt]
 ={\rm Prob}\left\{\left|\displaystyle \sum_{t=1}^T\displaystyle \sum_{i=1}^p\lambda^{t}_i\left(g_i(x^t)- G_i(x^t,\xi_t)\right)\right|\geq p\eta\widehat \phi (T,\eta) \rho_cT^{1/2}\right\}\\[10pt]
 = {\rm Prob}\left\{ \left|\displaystyle \sum_{t=1}^T\displaystyle \sum_{i=1}^p\lambda^{t}_i\left(g_i(x^t)- G_i(x^t,\xi_t)\right)  \right|\geq p\eta\widehat \phi (T,\eta) \rho_cT^{1/2}\,\Big|\, B\right\}{\rm Prob}(B)\\[10pt]
 \quad +{\rm Prob}\left\{ \left|\displaystyle \sum_{t=1}^T\displaystyle \sum_{i=1}^p\lambda^{t}_i\left(g_i(x^t)- G_i(x^t,\xi_t)\right)  \right|\geq p\eta\widehat \phi (T,\eta) \rho_cT^{1/2}\,\Big|\, B^c\right\}{\rm Prob}(B^c)\\[10pt]
 \leq   {\rm Prob}\left\{ \displaystyle \sum_{t=1}^T\displaystyle \sum_{i=1}^p\left|g_i(x^t) -G_i(x^t,\xi_t)  \right|  \geq \rho_c\eta T^{1/2}\right\}+{\rm Prob}(B^c)\\[15pt]
 \leq 2 \exp(-\eta^2/3)+\exp(-\eta).
 \end{array}
 \end{equation}
 The result comes from (\ref{eq:comineq0a}). \hfill $\Box$

 Based on Propositions \ref{th:constr3},
 \ref{th:KKT-res} and \ref{th:comp-res}, we obtain the following theorem about the high probability guarantees for $\varepsilon$-Karush-Kuhn
 -Tucker stationarity.
\begin{theorem}\label{th:high-prob}
Let $(x^t,\lambda^t)$ be generated by PMQSopt, the conditions in Theorem \ref{mainth:regrets} and Assumption C be satisfied. Then there exist constants $K_1>0$, $K_2>0$ and $K_3>0$ such that
\begin{equation}\label{eq:KKProbN}
\Pr\left[ \frac{1}{T}\sum_{t=1}^{T}\|R_{\alpha/2}(x^t,\lambda^t)\| \le K_1T^{-1/8} \right] \ge 1 - \frac{1}{T^{2/3}},
\end{equation}
\begin{equation}\label{eq:constrProbN}
{\rm Pr} \left[\displaystyle \frac{1}{T}\displaystyle \sum_{t=1}^Tg_i(x^t) \leq K_2 T^{-1/4}
 \right] \geq 1-\displaystyle \frac{1}{T}-\displaystyle \frac{1}{T^{2/3}}\geq 1-\displaystyle \frac{2}{T^{2/3}},
\end{equation}
and
\begin{equation}\label{eq:Com-ResProbN}
{\rm Pr} \left[\displaystyle \frac{1}{T}\Big|\displaystyle \sum_{t=1}^T\langle \lambda^t, g(x^t)\rangle\Big|\leq  K_3 T^{-1/4}\right] \geq 1-\displaystyle \frac{1}{T}-\displaystyle \frac{2}{T^{2/3}}\geq 1-\displaystyle \frac{3}{T^{2/3}}.
\end{equation}

\end{theorem}
\textbf{Proof}. First, we prove (\ref{eq:KKProbN}). By Proposition \ref{th:KKT-res}, we obtain
\begin{equation}\label{eq:b1}
\Pr\left[ \frac{1}{T}\sum_{t=1}^{T}\|\nabla\phi_{1/\alpha}^t(x^t)\|^2 \le \pi_{\text{grad}}\bigl(T,\tfrac{2}{3}\log T\bigr) \right] \ge 1 - \frac{1}{T^{2/3}}.
\end{equation}
Setting $\eta=\frac{2}{3}\log T$ yields $e^{-\eta}=T^{-2/3}$, which simplifies the probability bound to $1 - T^{-2/3}$. The explicit expression for $\widehat{\phi}(T,\eta)$ then becomes
\[
\widehat{\phi}\bigl(T,\tfrac{2}{3}\log T\bigr) =\kappa_0+\kappa_1+(\kappa_1\beta+\kappa_4) T^{-1/4}+\kappa_3 T^{-3/4}
+16\beta\displaystyle \frac{\gamma_1^2}{\varepsilon_0}\left(\tfrac{2}{3}\log T+\log (T+1)\right)T^{-1/4}.
\]
Clearly, the quantity $\widehat{\phi}\bigl(T,\tfrac{2}{3}\log T\bigr)$ is bounded by a constant independent of $T$; let us denote this constant by $C_0$. Substituting the explicit values
\[
\alpha+\tau = \beta T^{1/4} + T^{1/2},\,\, \alpha\sigma = \beta T^{-1/2},\,\, \alpha(\alpha+\tau)^{-1} = \frac{\beta T^{1/4}}{\beta T^{1/4}+T^{1/2}} \le \beta T^{-1/4},
\]
we obtain
\begin{align*}
\pi_{\text{grad}}\bigl(T,\tfrac{2}{3}\log T\bigr) &\le \frac{4(\beta T^{1/4}+T^{1/2})}{T}\Bigl[f(x^1)-\inf f\Bigr] + \frac{4\nu_g(\beta T^{1/4}+T^{1/2})}{T}C_0 \\
&\quad + 4(\beta T^{1/4}+T^{1/2})\nu_g\gamma_1 T^{-3/4} + 4\beta T^{-1/2} D_0\gamma_2 p[\kappa_g+\kappa_\Sigma D_0] \\
&\quad + 4\beta T^{-1/4}\bigl[\kappa_f + \gamma_2 p(\kappa_g+\kappa_\Sigma D_0)\bigr]^2 + 4\beta T^{-1/4}p\kappa_g^2 C_0^2.
\end{align*}
Observe that $(\beta T^{1/4}+T^{1/2})/T = \beta T^{-3/4} + T^{-1/2}$ and $(\beta T^{1/4}+T^{1/2})T^{-3/4} = \beta T^{-1/2} + T^{-1/4}$. Consequently, all terms are of order $T^{-1/4}$ or smaller. Collecting the dominant terms yields
\[
\pi_{\text{grad}}\bigl(T,\tfrac{2}{3}\log T\bigr) \le C_1 T^{-1/4}
\]
for some constant $C_1>0$. Inequality (\ref{eq:re-ms}) gives
\[
\|R_{\alpha/2}(x^t,\lambda^t)\| \le \frac{3}{2}\Bigl(1+\frac{1}{\sqrt{2}}\Bigr)\|\nabla\phi_{1/\alpha}^t(x^t)\|,
\]
and by applying Jensen's inequality, we have
\[
\frac{1}{T}\sum_{t=1}^{T}\|R_{\alpha/2}(x^t,\lambda^t)\| \le \frac{3}{2}\Bigl(1+\frac{1}{\sqrt{2}}\Bigr)\sqrt{\frac{1}{T}\sum_{t=1}^{T}\|\nabla\phi_{1/\alpha}^t(x^t)\|^2}.
\]
Defining
$$
K_1=\frac{3}{2}\Bigl(1+\frac{1}{\sqrt{2}}\Bigr)\sqrt{C_1},
$$
we obtain (\ref{eq:KKProbN}) from (\ref{eq:b1}).

Next, we prove (\ref{eq:constrProbN}). For $\eta=\log T$, we have $\widehat{\phi}(T,\eta)\le C_2$ for some $C_2>0$ and $\pi_c(T,\eta)\le K_2 T^{3/4}$, where
$$
K_2=\rho_1\beta^{-1}+\rho_2\beta^{-1}\nu_g+\eta \rho_c+ (1+\beta^{-1}\rho_2)C_2.
$$
Thus, (\ref{eq:constrProbN}) follows directly from (\ref{eq:constrProb}).

Finally, following a similar approach to the above, we can readily deduce (\ref{eq:Com-ResProbN}) from (\ref{eq:Com-ResProb}) for some constant $K_3>0$. \hfill $\Box$

\begin{corollary}\label{th:corhigh-prob}
Let $(x^t,\lambda^t)$ be generated by PMQSopt, and the conditions in Theorem \ref{mainth:regrets} be satisfied. Then
\begin{equation}\label{eq:KKProbNaa}
{\rm Pr}_{\xi_{[T]}} \left\{\mathbb E_{R}\left[\|R_{\alpha/2}(x^{R},\lambda^{R})\|\right]
\leq K_1\left(T^{-1/8}\right)\right\} \geq 1 - \frac{1}{T^{2/3}},
\end{equation}
\begin{equation}\label{eq:constrProbNaa}
{\rm Pr}_{\xi_{[T]}} \left\{\mathbb E_{R}g_i(x^{R}) \leq K_2\left(T^{-1/4}\right)
 \right\} \geq 1-\displaystyle \frac{1}{T}-\displaystyle \frac{1}{T^{2/3}}\geq 1-\displaystyle \frac{2}{T^{2/3}},
\end{equation}
and
\begin{equation}\label{eq:Com-ResProbNaa}
{\rm Pr}_{\xi_{[T]}} \left\{\Big |\mathbb E_{R}\langle \lambda^{R}, g(x^{R})\rangle\Big|\leq K_3\left(T^{-1/4}\right)\right\} \geq 1-\displaystyle \frac{1}{T}-\displaystyle \frac{2}{T^{2/3}}
\geq 1-\displaystyle \frac{3}{T^{2/3}}.
\end{equation}
\end{corollary}
\section{Numerical Results}\label{Sec6}
\setcounter{equation}{0}
In this section, we evaluate the proposed proximal method of multipliers with quadratic approximations on three classes of stochastic constrained optimization problems. The first experiment is a synthetic quadratically constrained nonconvex problem (QCNP), which is used to illustrate the sample efficiency of PMQSopt and the ${\cal O}(T^{-1/4})$ trend predicted by our analysis. The second experiment is a nonconvex Neyman-Pearson classification problem, which represents a probability-constrained learning task on real data. The third experiment is a nonconvex fairness constrained classification problem, which tests the applicability of PMQSopt to constrained learning models with fairness requirements. All numerical experiments are carried out using MATLAB R2021a on a laptop with Intel i7-12700H CPU and 8.0GB of RAM. In the comparison figures, the horizontal axis is the cumulative number of stochastic gradient evaluations, which provides a hardware-independent measure of stochastic work.
\subsection{Solving subproblems}
This subsection describes the numerical solution of the PMQSopt subproblem, which can be written at iteration $t$ as
 \begin{equation}\label{xna0a}
\begin{array}{l}
x^{t+1}= \displaystyle\arg\min_{x\in X_0}\left\{ {\cal L}^t_{\sigma }(x,\lambda^t) +\displaystyle\frac{\alpha}{2}\|x-x^t\|^2\right\}.
\end{array}
\end{equation}
This problem can be reformulated in the following general form:
\begin{equation}\label{eq:general-subp}
\begin{array}{ll}
\displaystyle   \min_{x \in X_0}  \phi(x) := & c^T_0(x-x^t)+\displaystyle \frac{1}{2}\langle G_0 (x-x^t),x-x^t \rangle\\[10pt]
& +
\displaystyle \frac{1}{2\sigma}\displaystyle \sum_{i=1}^p\left[q_i+\sigma  \left(c^T_i(x-x^t)+\displaystyle \frac{1}{2}\langle G_i (x-x^t),x-x^t \rangle\right)\right]_+^2+\displaystyle \frac{\alpha}{2}\|x-x^t\|^2,
\end{array}
\end{equation}
with
$$
c_0=\nabla_x F(x^t,\xi_t),\quad q_i=\lambda^t_i+\sigma G_i(x^t,\xi_t), \quad c_i =\nabla_x G_i(x^t,\xi_t),\quad  G_i=\Sigma^t_i,\quad i=1,\ldots,p.
$$
Problem (\ref{eq:general-subp}) is strongly convex under the parameter choices used by PMQSopt. We solve it by Nesterov's accelerated projected gradient method (APG); see \cite{Beck2017}.

\begin{algorithm}[H]\label{Algorithm_sub}
	\textbf{Step 0}: Input $x^{0} \in X_{0}$ and $\eta >1$. Set $y^{0} = x^{0}$, $L_{-1} = 1$ and $k:=0$.
	
	\textbf{Step 1}: Set
	\begin{equation}\nonumber
		x^{k+1} = T_{L_{k}} (y^{k}),
	\end{equation}
	where $T_{L} (y) := \Pi_{X_{0}} \left[y - \frac{1}{L}\nabla \phi (y)\right]$, the stepsize $L_{k} = L_{k-1}\eta^{i_{k}}$, and $i_{k}$ is the smallest nonnegative integer satisfying the following condition
	\begin{equation}\nonumber
		\begin{aligned}
			\phi \left(T_{L_{k-1}\eta ^{i_{k}}}(y^{k})\right) \leq \phi(y^{k}) &+ \langle \nabla \phi (y^{k}), T_{L_{k-1}\eta ^{i_{k}}}(y^{k}) -y^{k} \rangle \\
			&+ \frac{L_{k-1}\eta ^{i_{k}}}{2} \Vert T_{L_{k-1}\eta ^{i_{k}}}(y^{k}) -y^{k} \Vert^2.
		\end{aligned}
	\end{equation}
	
	\textbf{Step 2}: Compute
	\begin{equation}\nonumber
		y^{k+1} = x^{k+1} +\frac{k}{k+3}\left(x^{k+1}-x^{k}\right).
	\end{equation}
	
	\textbf{Step 3}: Set $k :=k+1$ and go to Step 1.
	\caption{Nesterov's accelerated projected gradient method for the subproblem}
\end{algorithm}

Standard convergence guarantees for projected APG applied to smooth convex problems can be found in \cite{Beck2017}; stronger rates are available for strongly convex variants or restarted implementations. In our experiments, $X_0$ is chosen to be simple so that the projection $\Pi_{X_{0}}$ can be computed efficiently.

\subsection{Synthetic quadratically constrained nonconvex problem}
We first test PMQSopt on a synthetic quadratically constrained nonconvex problem (QCNP). This experiment serves as the main synthetic benchmark and is designed to illustrate both practical sample efficiency and the empirical behavior of the residuals appearing in our theory. Its construction is inspired by our earlier MATLAB implementation. In particular, we retain a common active-boundary structure so that all constraints share a similar geometry, while allowing moderate heterogeneity and quadratic perturbations to preserve nonconvexity. The problem is defined on the box $X_0:=\{x\in\R^n:\|x\|_{\infty}\le R\}$ and has the form
\begin{equation}
	\begin{aligned}
		\min_{x \in X_0}\quad & f(x):=\frac{1}{N}\sum_{s=1}^{N}\log\left(1+\frac{1}{2}\|H_sx-c_s\|^2\right)\\
		\mathrm{s.t.}\quad & g_i(x):=\frac{1}{N}\sum_{s=1}^{N}\left[a_{i,s}^\top(x-\bar x)+\frac{1}{2}(x-\bar x)^\top Q_{i,s}(x-\bar x)\right]\le 0,\quad i=1,\ldots,p,
	\end{aligned}
\end{equation}
where each $Q_{i,s}$ is diagonal. In our implementation, we use $n=50$, $p=50$, $N=100$, $m=5$, and $R=10$. We first sample a boundary reference point $\bar x\in[-1.5,-0.5]^n$ componentwise and then define a strictly feasible point by $x^{\rm feas}:=\bar x-0.5\mathbf{1}_n$. The objective is centered at a point $x^{\rm obj}=\mathbf{1}_n$ through the choice $c_s=H_sx^{\rm obj}$, where each $H_s\in\R^{m\times n}$ has i.i.d. entries ${\cal N}(0,1/n)$. For the constraints, the diagonal entries of $Q_{i,s}$ are drawn uniformly from $[-q_{\max},q_{\max}]$ with $q_{\max}=0.05$, while each component of $a_{i,s}$ is drawn from $[0.5,0.7]$. Thus, every constraint has a similar first-order active direction at $\bar x$, but the coefficients are not identical, and the quadratic perturbations introduce moderate nonconvexity. By construction, one has $g_i(\bar x)=0$ for every $i$, whereas $g_i(x^{\rm feas})<0$ for all $i$. Therefore, the problem possesses a common active boundary point together with a strictly feasible point, which is consistent with the assumptions used in our analysis.

For this synthetic QCNP, we compare PMQSopt with MLALM \cite{Shi2025} and Stoc-iALM \cite{Li2024}. The comparison is averaged over 10 independent runs and is reported against the cumulative number of stochastic gradients. We plot the current objective value and the current feasibility error $\sum_{i=1}^{p}[g_i(x^t)]_+$, since these two quantities provide a direct view of practical progress on this problem. The results are shown in Figures \ref{fig:qcnp-compare-obj} and \ref{fig:qcnp-compare-cons}. In these tests, PMQSopt achieves a competitive objective decrease while maintaining steady feasibility improvement. In particular, PMQSopt reaches a low-objective and near-feasible regime with fewer stochastic gradients than MLALM, while using less stochastic work than Stoc-iALM. These observations are consistent with the sample-efficient design of PMQSopt.

Besides the algorithmic comparison, we also examine whether the empirical decay of the theory-driven residuals is consistent with the complexity bounds established in Section~\ref{Sec3}. To this end, we track the two quantities
\[
r_{\rm KKT}^{2}(T):=\frac{1}{T}\sum_{t=1}^{T}\|R_{\alpha_{\rm met}}(x^t,\lambda^t)\|^2,\quad
r_{\rm cons}(T):=\frac{1}{T}\sum_{t=1}^{T}\sum_{i=1}^{p}[g_i(x^t)]_+.
\]
According to Theorem \ref{mainth:regrets}, these residuals are predicted to decay at the order ${\cal O}(T^{-1/4})$. Figures \ref{fig:qcnp-theory-kkt} and \ref{fig:qcnp-theory-cons} therefore plot the empirical averages, computed over 8 independent runs, together with a power-law fit and a reference curve proportional to $T^{-1/4}$. The squared KKT residual displays a clear decreasing trend that is close to the predicted order. The constraint residual also decreases as $T$ grows, although its empirical slope is somewhat flatter than the ideal $-1/4$ benchmark on this instance. Overall, these results are consistent with the theoretical trend predicted for PMQSopt on synthetic QCNPs. To keep the presentation focused, we do not include the complementarity residual in this subsection.
\begin{figure}[!htb]
	\centering
	\subfloat[Objective value.]{\expplot{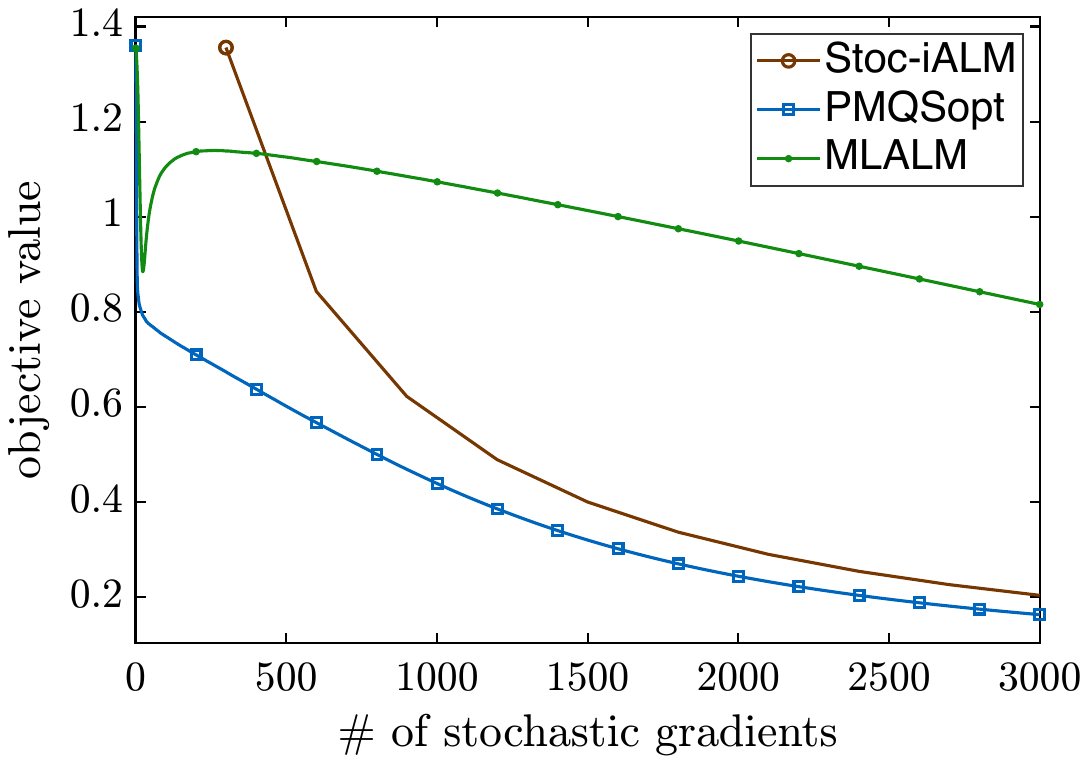}
	\label{fig:qcnp-compare-obj}}
	\hfill
	\subfloat[Constraint violation.]{\expplot{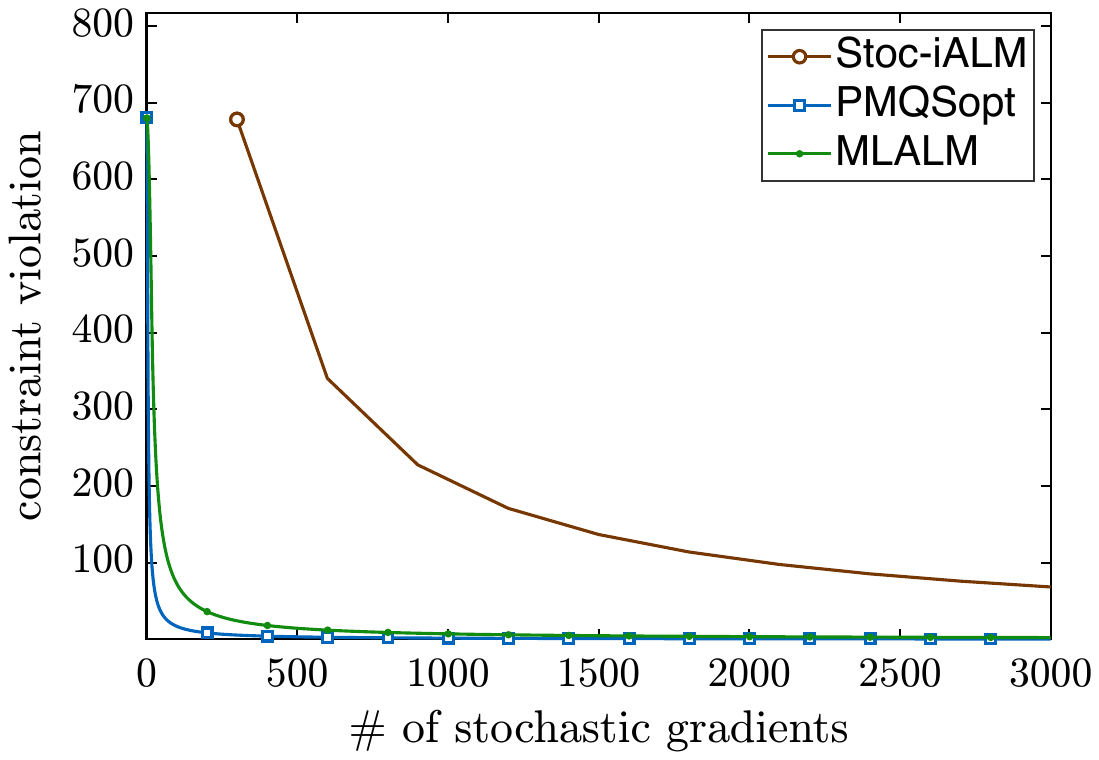}
	\label{fig:qcnp-compare-cons}}
	\caption{Objective value and constraint violation versus the number of stochastic gradients on the synthetic QCNP.}
\end{figure}

\begin{figure}[!htb]
	\centering
	\subfloat[Squared KKT residual.]{\expplot{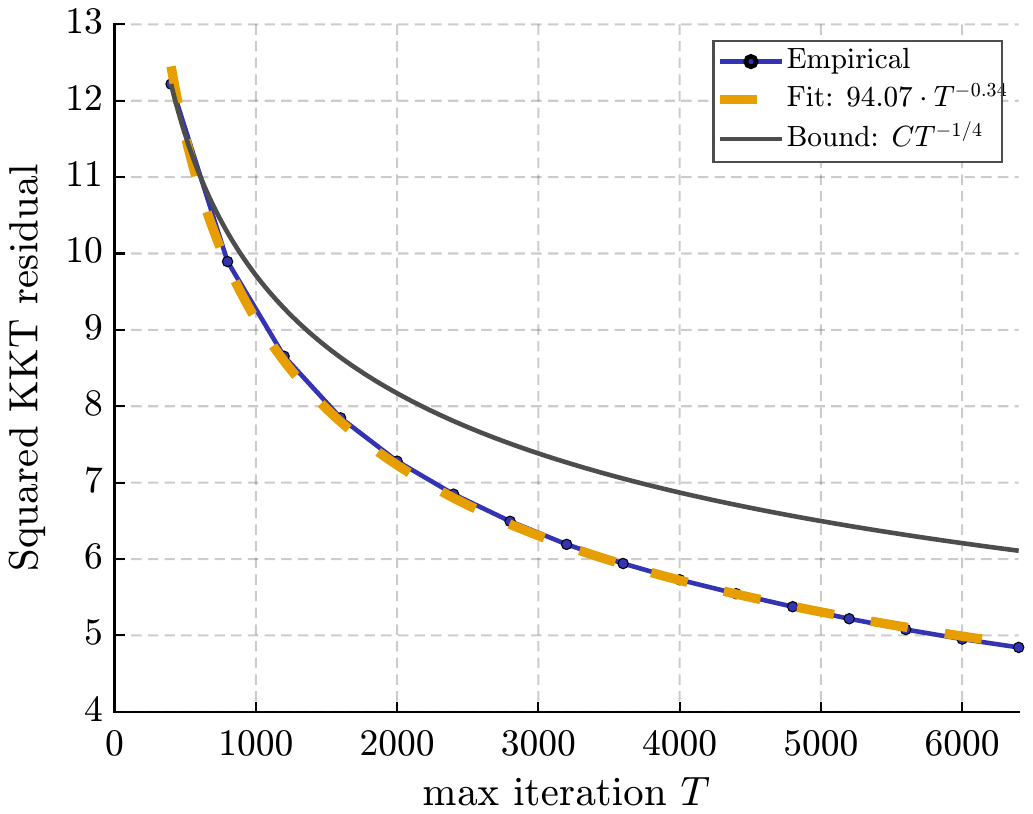}
	\label{fig:qcnp-theory-kkt}}
	\hfill
	\subfloat[Constraint residual.]{\expplot{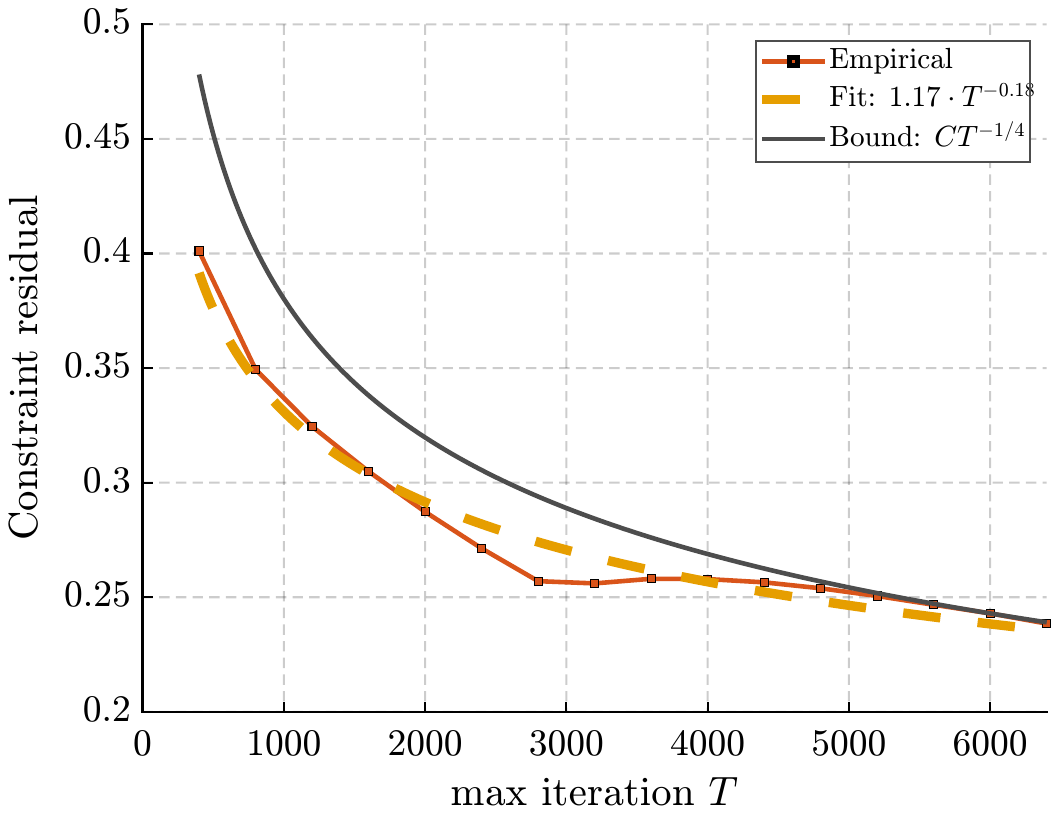}
	\label{fig:qcnp-theory-cons}}
	\caption{Empirical decay of the theory-driven residuals for PMQSopt on the synthetic QCNP.}
\end{figure}

\FloatBarrier
\subsection{Nonconvex Neyman-Pearson classification}
In this subsection, we test PMQSopt on a nonconvex Neyman-Pearson classification problem. This experiment represents a real-data probability constrained classification task. The goal is to minimize the false-negative error while controlling the false-positive error at a prescribed level. The model can be written as
\begin{equation}
	\begin{aligned}
		\min_{x \in \R^{d}} \quad &f(x) := \frac{1}{N_{0}} \sum_{i=1}^{N_{0}} \phi (x^{\top}a_{i}^{0}) \\
		\mathrm{s.t.} \quad & g(x) := \frac{1}{N_{1}} \sum_{i=1}^{N_{1}} \phi(- x^{\top}a_{i}^{1}) - \tau \leq 0,
	\end{aligned}
\end{equation}
where $\{a_{i}^{0}\}_{i=1}^{N_{0}}$ and $\{a_{i}^{1}\}_{i=1}^{N_{1}}$ denote the positive-class and negative-class samples in the training set, respectively. The parameter $\tau$ specifies the admissible false-positive level, and $\phi(u)=1/(1+\exp(u))$ is the sigmoid loss.

The data sets used in our comparison are listed in Table \ref{table 1}. For multi-class data, we construct a binary Neyman-Pearson task; for instance, on MNIST we classify odd digits versus even digits. The MATLAB scripts support all three data sets, and we report the numerical results on MNIST, CINA, and gisette in the main text.
\begin{table}
	\centering
	\caption{Data sets used in Neyman-Pearson classification}
	\label{table 1}
	\begin{tabular}{cccccc}
		\toprule[2pt]
		Dataset & Data $N$ & Variable $n$ & Density & False-positive level $\tau$  & Reference\\
		\midrule[1pt]
		MNIST & 60000 & 784 & 19.12\% & 0.2 & LeCun et al.\cite{LeCun2010}\\
		CINA & 16033 & 132 & 29.56\% & 0.3 & workbench team\cite{CINA2008}\\
		gisette & 6000 & 5000 &12.97\%  & 0.2 & Guyon et al.\cite{Guyon2004}\\
		\bottomrule[2pt]
	\end{tabular}
\end{table}

We compare PMQSopt with MLALM \cite{Shi2025} and Stoc-iALM \cite{Li2024}. Stoc-iALM is developed for a related nonconvex expectation-constrained setting, and we use the standard slack-variable reformulation when needed to apply it to the present inequality-constrained model. Since all three methods are augmented-Lagrangian type schemes, these comparisons are natural from both the modeling and algorithmic viewpoints.

For all three methods, we use mini-batch sampling with batch size $10$ for both the objective and the constraint estimators, and we average the reported curves over 10 independent runs. In PMQSopt, the stochastic gradient and Hessian estimates at iteration $t$ are formed from
\begin{equation*}
	\begin{aligned}
	c_{0}^{t} &= \frac{1}{\lvert \mathcal{N}_{0}^{t}\rvert} \sum_{i \in \mathcal{N}_{0}^{t}} \nabla f_{i}(x^t),\quad c_{1}^{t} := \frac{1}{\lvert \mathcal{N}_{1}^{t} \rvert } \sum_{i \in \mathcal{N}_{1}^{t}} \nabla g_{i}(x^t), \\
	\Sigma_{0}^{t} &= \frac{1}{\lvert \mathcal{N}_{0}^{t}\rvert} \sum_{i \in \mathcal{N}_{0}^{t}} \nabla^{2} f_{i}(x^t),\quad \Sigma_{1}^{t} := \frac{1}{\lvert \mathcal{N}_{1}^{t} \rvert } \sum_{i \in \mathcal{N}_{1}^{t}} \nabla^{2} g_{i}(x^t),
    \end{aligned}
\end{equation*}
where $f_{i}(x)=\phi(x^{\top}a_{i}^{0})$ and $g_{i}(x)=\phi(-x^{\top}a_{i}^{1})$, and the index sets $\mathcal{N}_{0}^{t}$ and $\mathcal{N}_{1}^{t}$ are sampled uniformly from $\{1,\ldots,N_{0}\}$ and $\{1,\ldots,N_{1}\}$, respectively. For PMQSopt, we use the same practical finite-horizon parameter choice as in our MATLAB prototype, namely $\alpha=\mathcal{O}(\sqrt{T})$ and $\sigma=\mathcal{O}(T^{-1/2})$. The implementations of MLALM and Stoc-iALM use the same mini-batch sizes under their respective update rules.

For each data set, we report two comparison plots: objective value versus the cumulative number of stochastic gradients and constraint violation versus the cumulative number of stochastic gradients. This choice makes the stochastic work used by the three methods directly comparable and emphasizes the sample efficiency of the algorithms.

Figures \ref{fig:np-mnist-obj-grad}--\ref{fig:np-gisette-cons-grad} summarize the numerical behavior on the three data sets. Across these tests, PMQSopt attains competitive objective values while keeping the Neyman-Pearson constraint under control. Compared with MLALM, PMQSopt often reaches a lower objective value with a smaller stochastic-gradient budget, although MLALM can be more aggressive in reducing the feasibility error on some instances. Relative to Stoc-iALM, PMQSopt typically requires fewer stochastic gradient evaluations to reach the same level of objective decrease. Overall, the results indicate that PMQSopt provides a favorable objective-feasibility tradeoff for nonconvex Neyman-Pearson classification.

\begin{figure}[!htb]
	\centering
	\subfloat[Objective value.]{\expplot{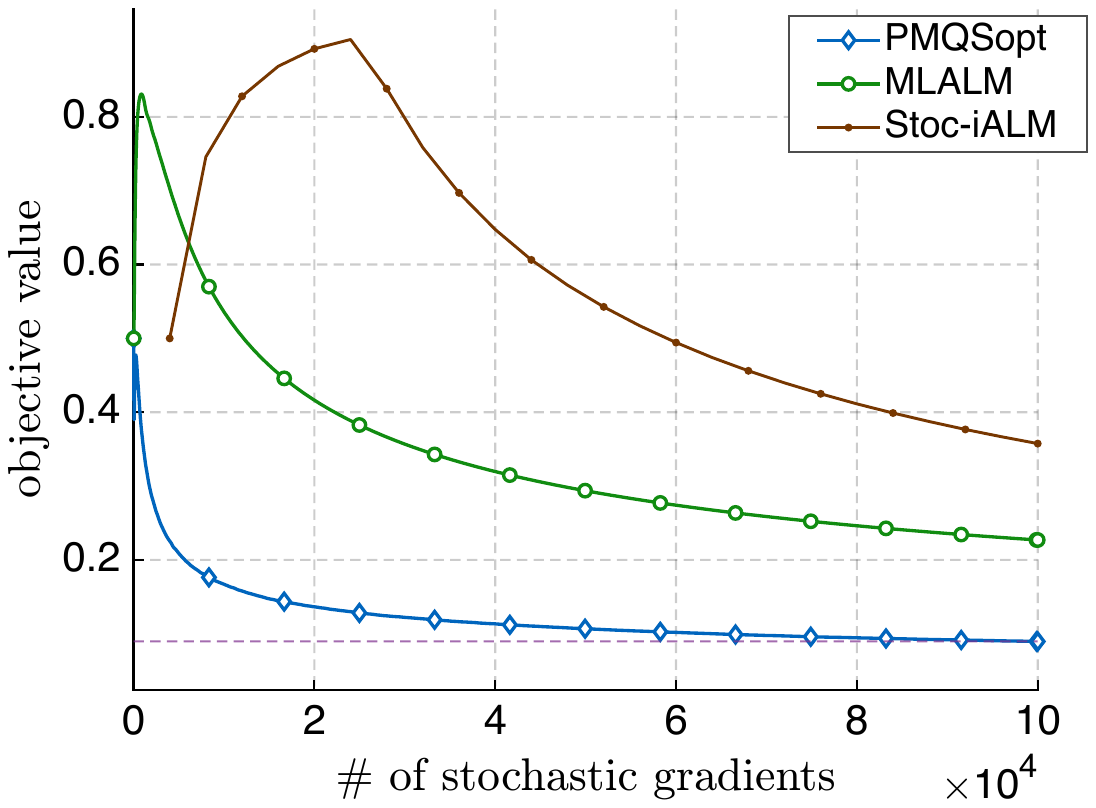}
	\label{fig:np-mnist-obj-grad}}
	\hfill
	\subfloat[Constraint violation.]{\expplot{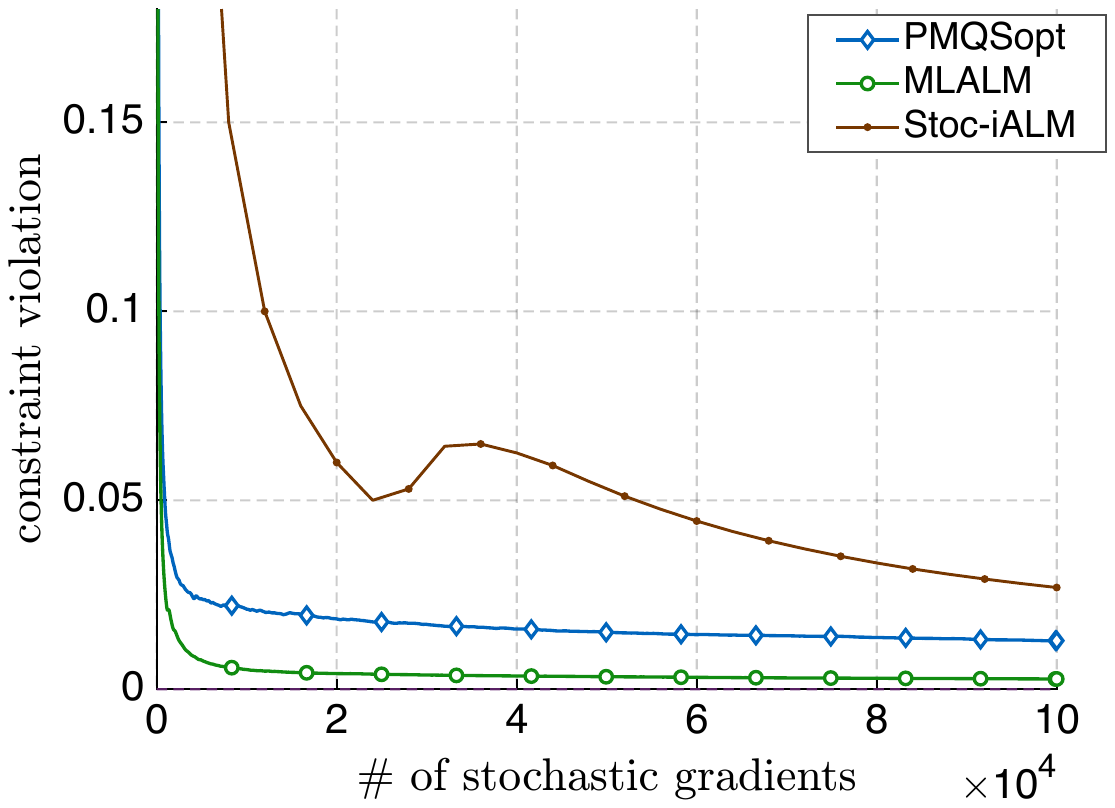}
	\label{fig:np-mnist-cons-grad}}
	\caption{Objective value and constraint violation versus the number of stochastic gradients on MNIST for nonconvex Neyman-Pearson classification.}
\end{figure}

\begin{figure}[!htb]
	\centering
	\subfloat[Objective value.]{\expplot{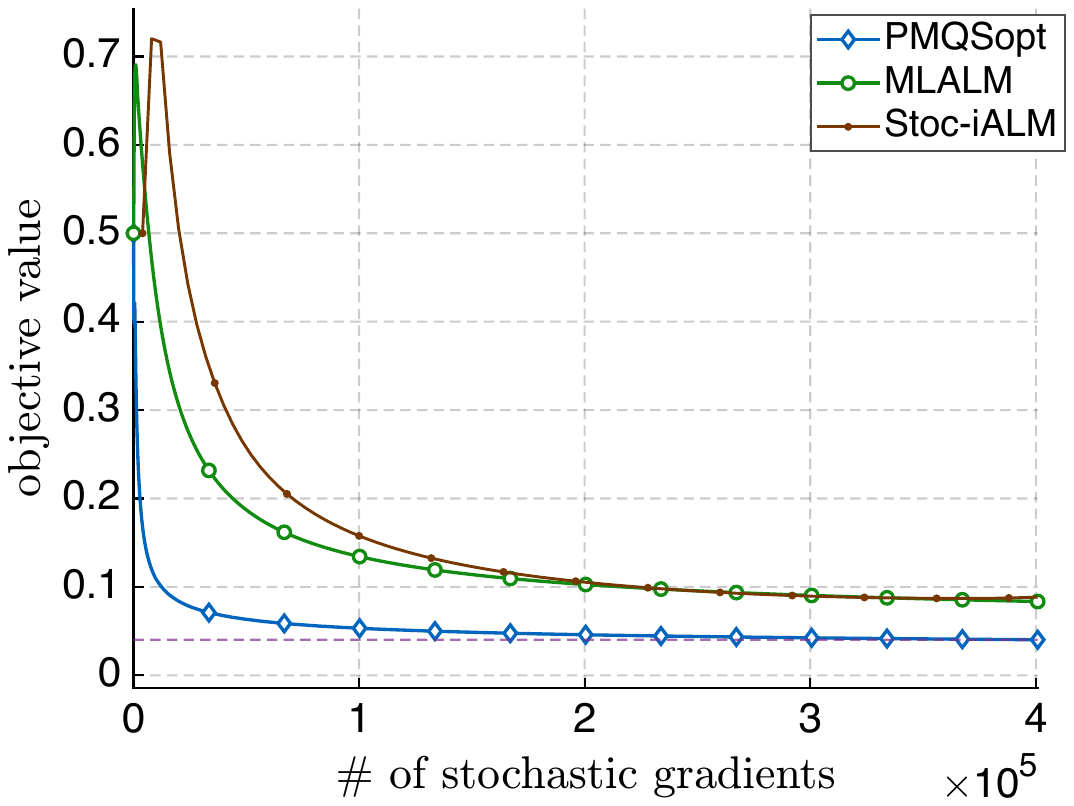}
	\label{fig:np-cina-obj-grad}}
	\hfill
	\subfloat[Constraint violation.]{\expplot{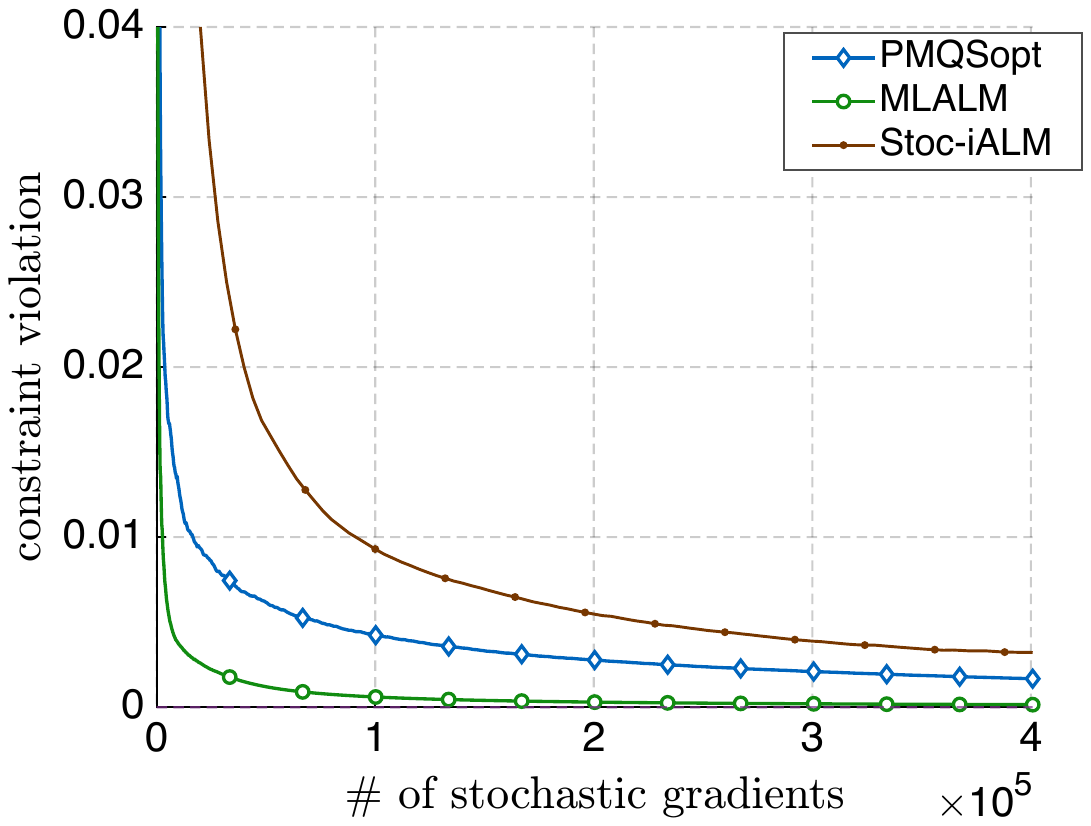}
	\label{fig:np-cina-cons-grad}}
	\caption{Objective value and constraint violation versus the number of stochastic gradients on CINA for nonconvex Neyman-Pearson classification.}
\end{figure}

\begin{figure}[!htb]
	\centering
	\subfloat[Objective value.]{\expplot{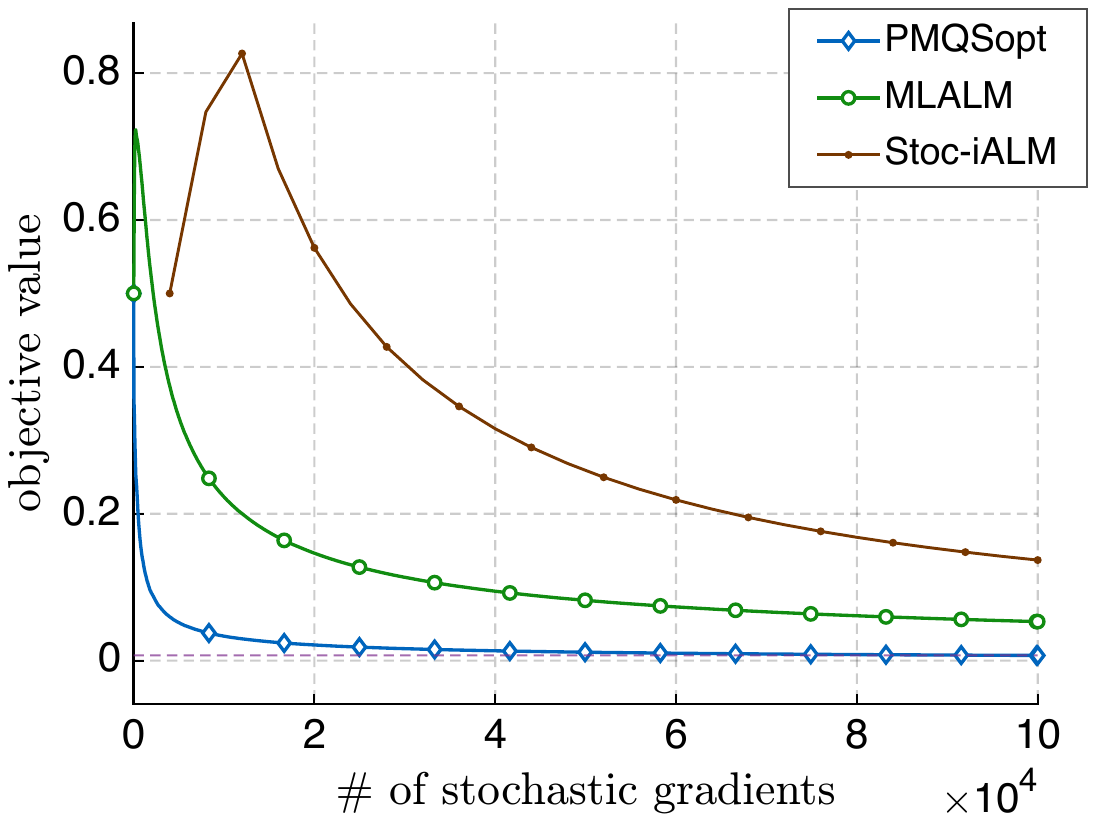}
	\label{fig:np-gisette-obj-grad}}
	\hfill
	\subfloat[Constraint violation.]{\expplot{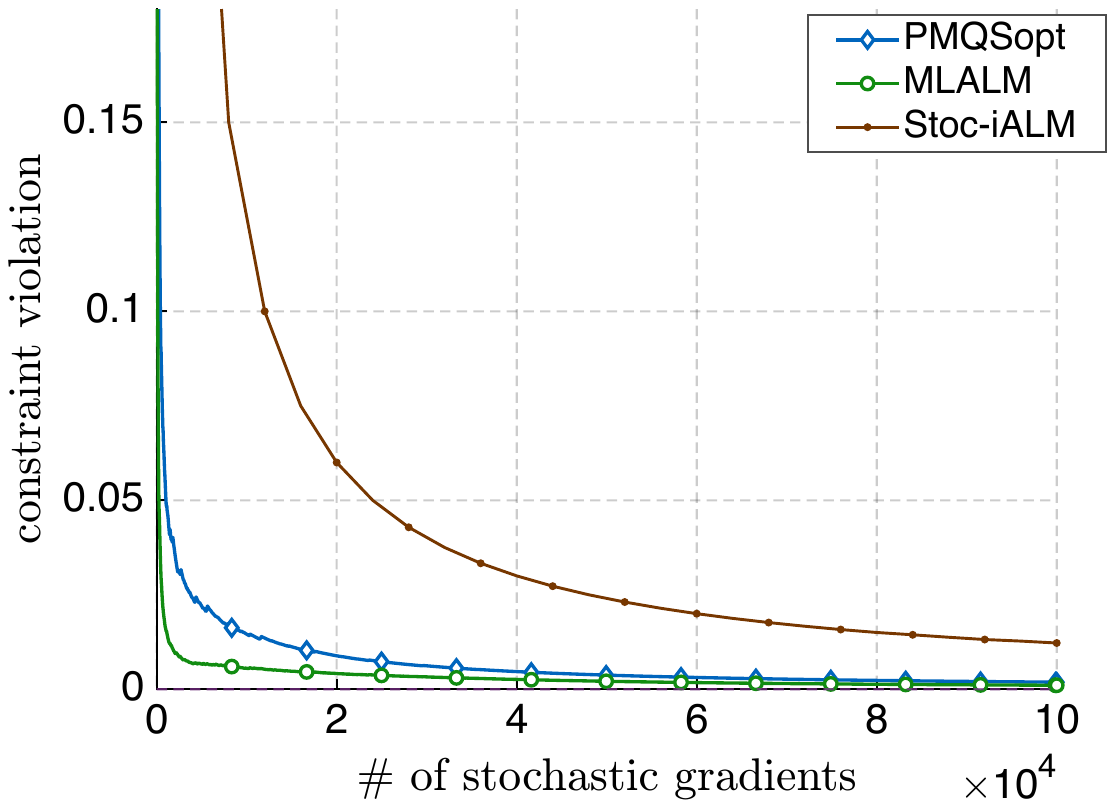}
	\label{fig:np-gisette-cons-grad}}
	\caption{Objective value and constraint violation versus the number of stochastic gradients on gisette for nonconvex Neyman-Pearson classification.}
\end{figure}

\FloatBarrier
\subsection{Nonconvex fairness constrained classification}
We next consider a nonconvex fairness constrained classification problem. This experiment examines the applicability of PMQSopt to nonconvex constrained learning models with a fairness-type constraint. Let $D$ denote the whole training set, let $S$ denote the samples belonging to a protected group, and let $S_{\min}\subseteq S$ denote the minority subgroup. The model used in the experiment is
\begin{equation}
	\begin{aligned}
		\min_{x}\quad & f(x):=\frac{1}{|D|}\sum_{(a,b)\in D}\phi_{\alpha}\big(\log(1+\exp(-b a^\top x))\big)\\
		\mathrm{s.t.}\quad & g(x):=c\frac{1}{|S|}\sum_{a\in S}\sigma(a^\top x)
		-\frac{1}{|S_{\min}|}\sum_{a\in S_{\min}}\sigma(a^\top x)\le 0,
	\end{aligned}
\end{equation}
where $\sigma(u)=1/(1+\exp(-u))$, $\phi_{\alpha}(u)=\alpha\log(1+u/\alpha)$ is the truncated logistic-loss transformation used in the MATLAB implementation with $\alpha=2$, $c=\tau |S|/|S_{\min}|$, and $\tau$ controls the admissible fairness violation level. This experiment uses the same three algorithms as above. The mini-batch sizes are $30$ for both the objective and constraint estimators, and the reported curves are averaged over 5 independent runs. The data sets are a9a, student-por, and student-mat, with $\tau=0.1$, $\tau=0.62$, and $\tau=0.55$, respectively.

Figures \ref{fig:fairness-a9a-grad}--\ref{fig:fairness-student-mat-grad} show the objective value and fairness constraint violation with respect to the cumulative number of stochastic gradients. On these data sets, PMQSopt typically attains the lowest or comparable final objective value among the compared methods while keeping the fairness violation at a comparable scale. MLALM can be more aggressive in reducing the fairness violation on some instances, whereas Stoc-iALM tends to produce more conservative but slower objective decrease. Overall, the results suggest that PMQSopt provides a favorable objective-feasibility tradeoff under the same stochastic-gradient budget.

\begin{figure}[!htb]
	\centering
	\subfloat[Objective value.]{\expplot{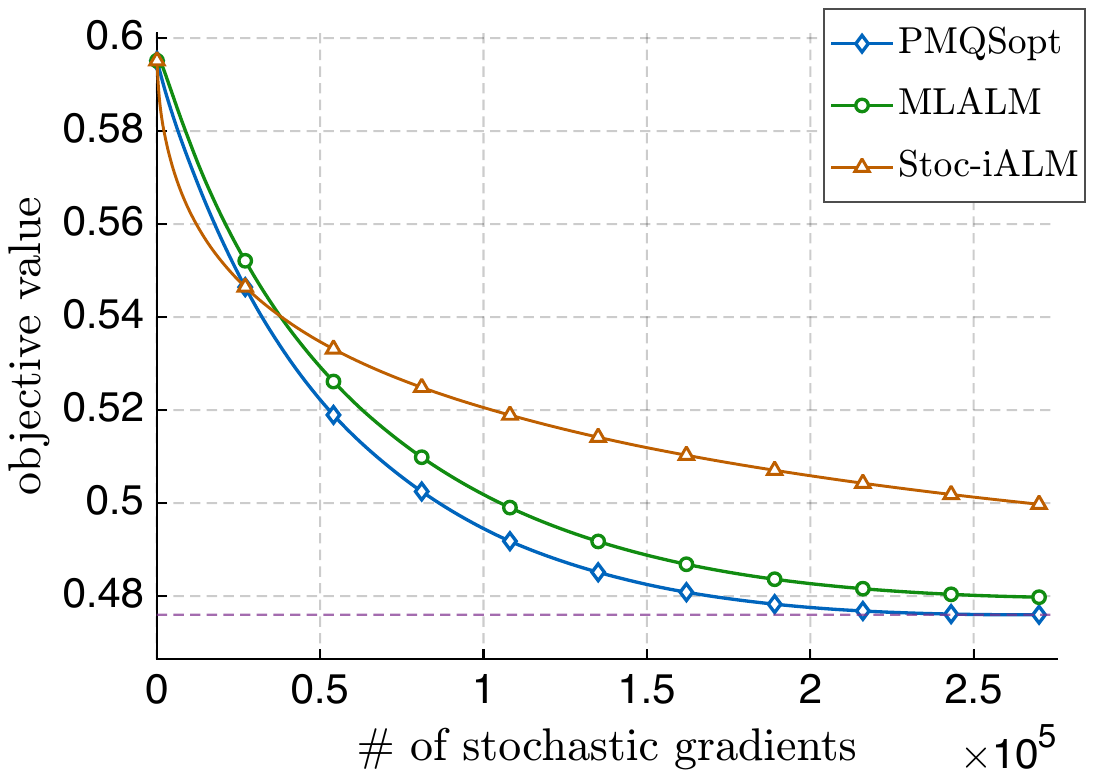}}
	\hfill
	\subfloat[Constraint violation.]{\expplot{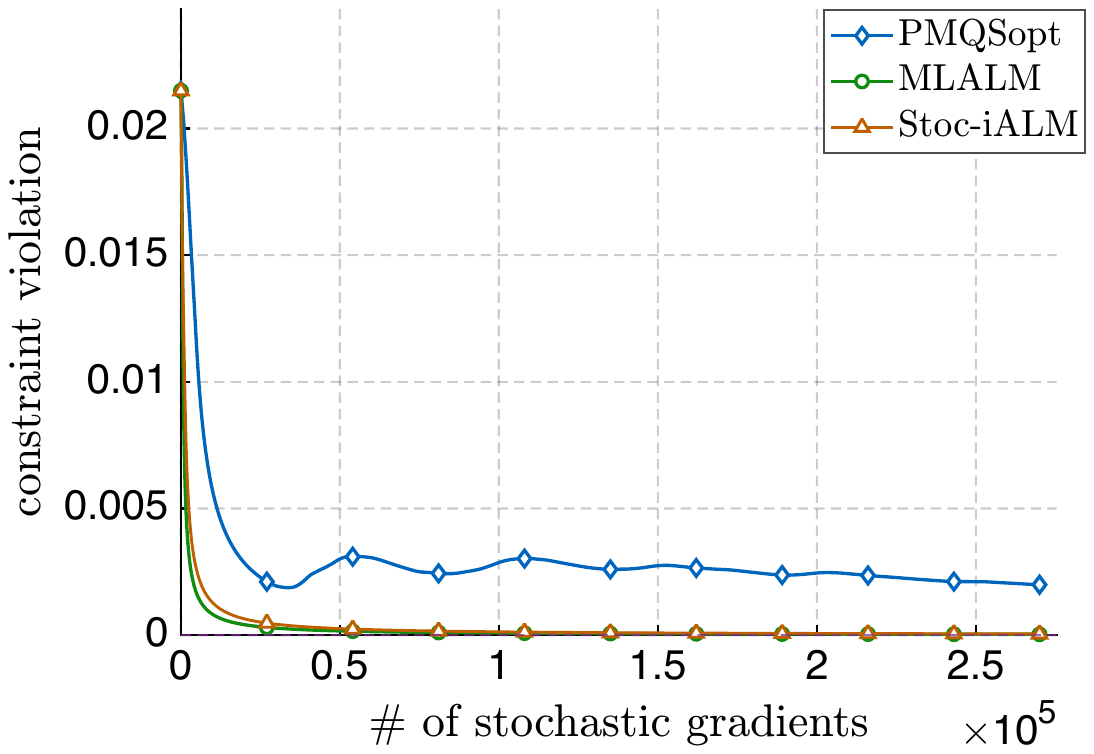}}
	\caption{Objective value and fairness constraint violation versus the number of stochastic gradients on the a9a data set.}
	\label{fig:fairness-a9a-grad}
\end{figure}

\begin{figure}[!htb]
	\centering
	\subfloat[Objective value.]{\expplot{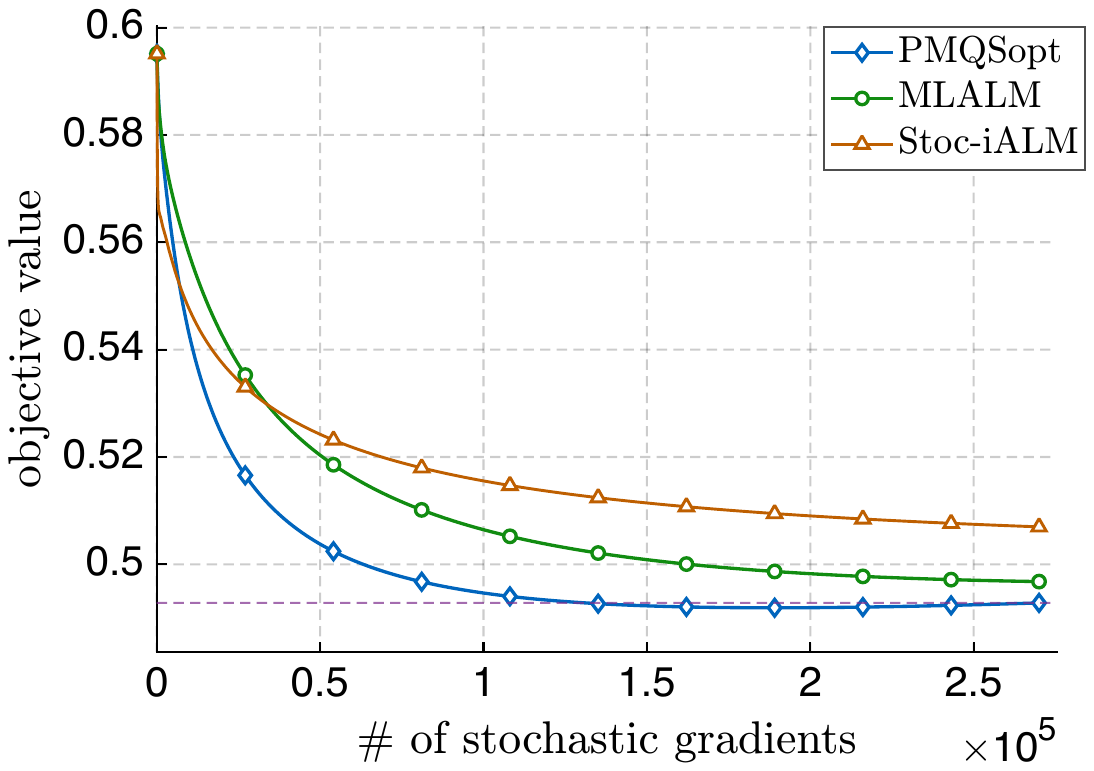}}
	\hfill
	\subfloat[Constraint violation.]{\expplot{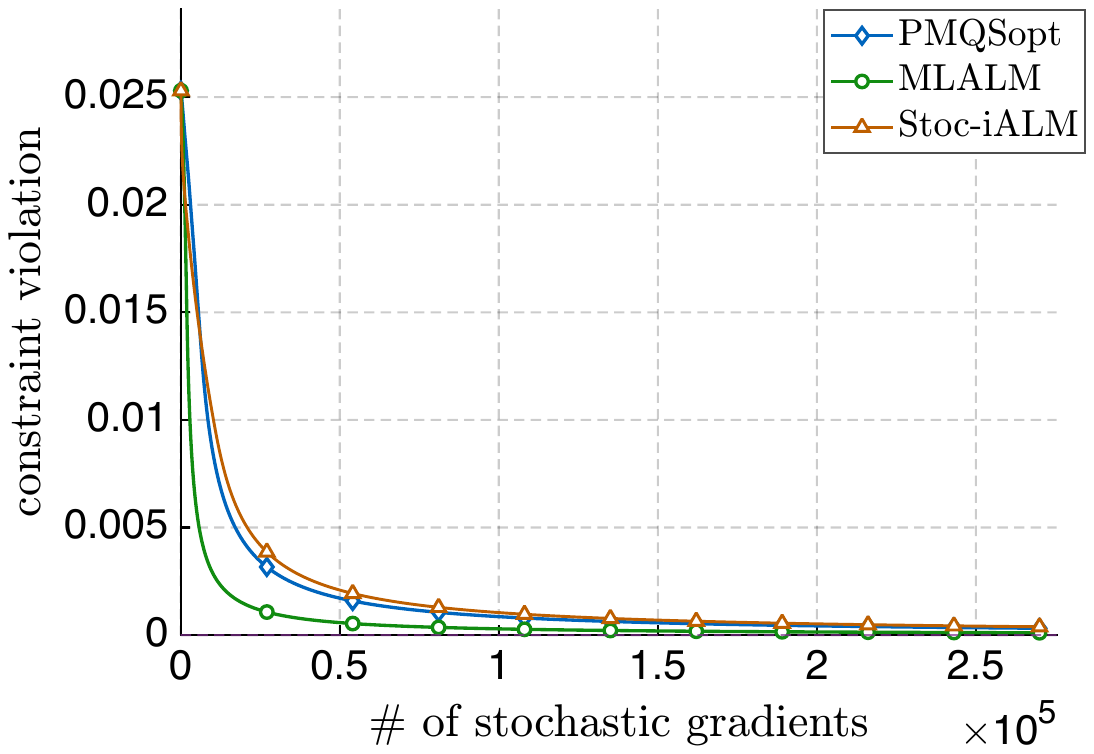}}
	\caption{Objective value and fairness constraint violation versus the number of stochastic gradients on the student-por data set.}
	\label{fig:fairness-student-por-grad}
\end{figure}

\begin{figure}[!htb]
	\centering
	\subfloat[Objective value.]{\expplot{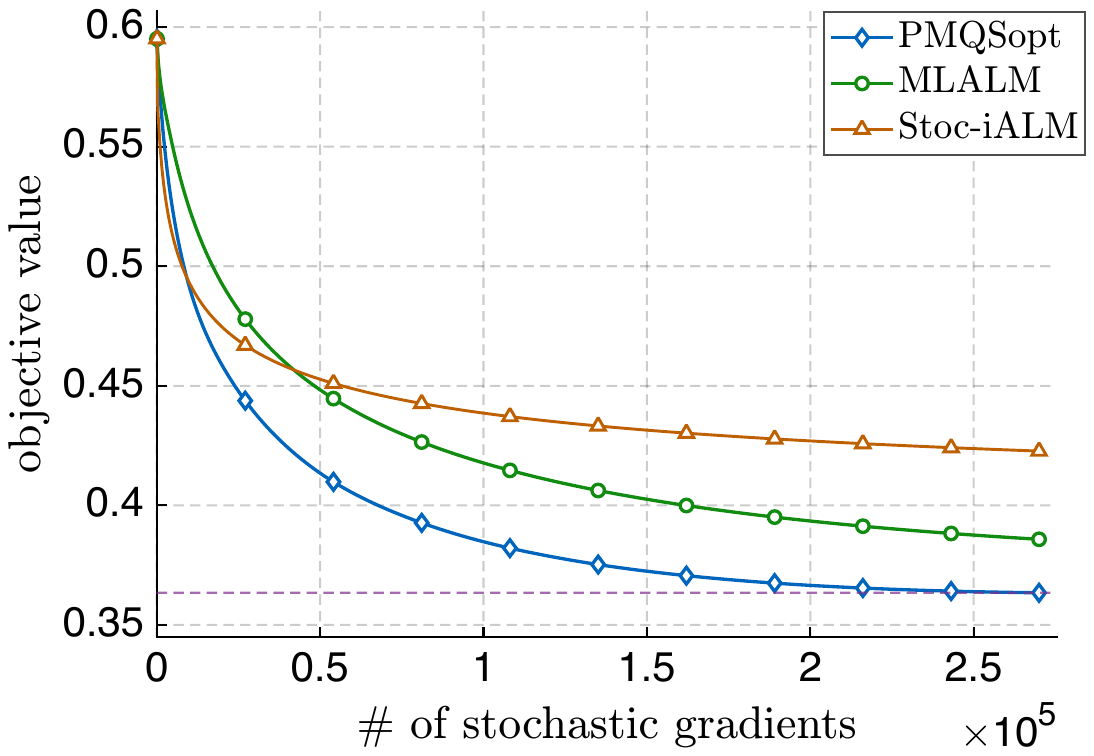}}
	\hfill
	\subfloat[Constraint violation.]{\expplot{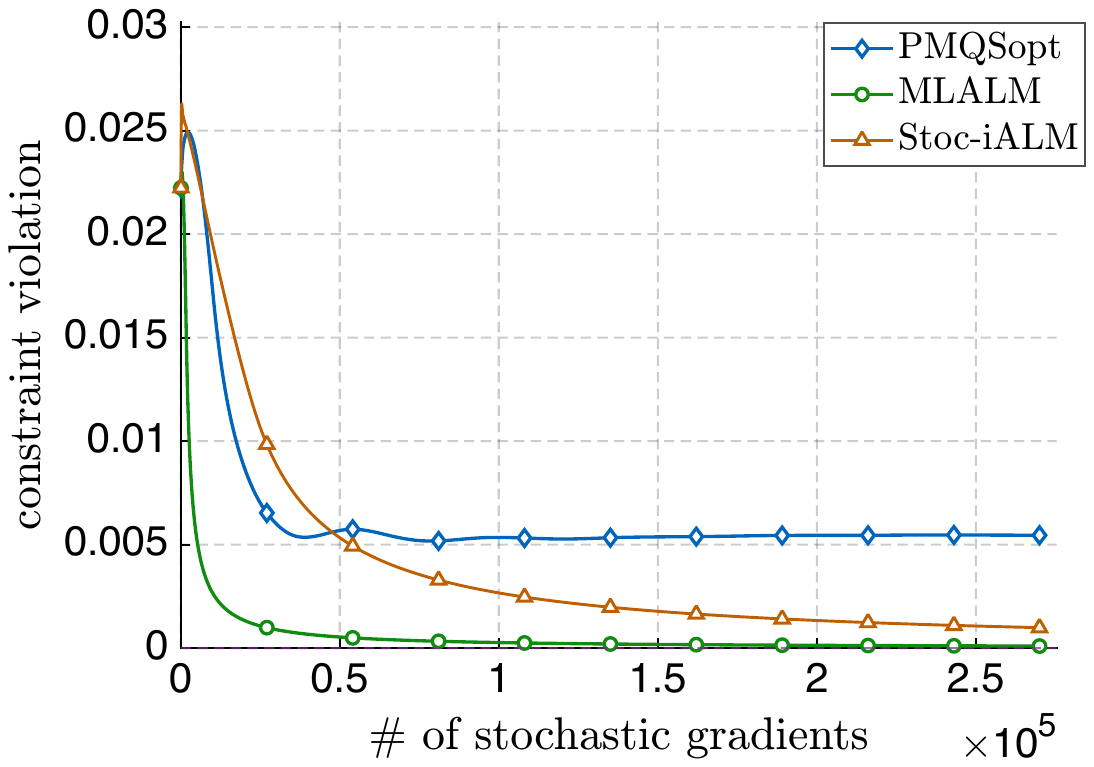}}
	\caption{Objective value and fairness constraint violation versus the number of stochastic gradients on the student-mat data set.}
	\label{fig:fairness-student-mat-grad}
\end{figure}

\FloatBarrier

\section{Conclusion}\label{Sec7}
\setcounter{equation}{0}
In this paper, we present a stochastic approximation method for solving nonconvex constrained stochastic optimization problems defined by expectations. The proposed method is based on the proximal method of multipliers applied to quadratic approximations of the original problem. We show that, when the objective and constraint functions are weakly convex, the algorithm achieves an expected convergence rate of ${\rm O}(T^{-1/4})$ for the average squared norm of the gradient of the Moreau envelope of the Lagrangian, the average constraint violation, and the average complementarity violation. Moreover, we demonstrate that, with high probability, PMQSopt attains ${\rm O}(T^{-1/4})$ bounds for the Lagrangian gradient violation, constraint violation, and complementarity violation. When the constraint functions are convex and the subproblem in PMQSopt is solved via its dual, the algorithm can be implemented as a practical projection method for the stochastic optimization problem. Numerical results illustrate its competitiveness on the tested instances.

A key advantage of the proposed PMQSopt algorithm is that its theoretical guarantees are established under only two mild conditions: (i) weak convexity of all problem functions, and (ii) the existence of a strictly feasible point. Furthermore, the algorithm is designed as a sequentially strongly convex programming method, ensuring that its subproblems can be solved efficiently.

It should be noted that, under different settings, \cite{Boob2023} established an $\mathcal{O}(T^{-3})$ complexity for achieving an $(\varepsilon,\delta)$-KKT point, while \cite{Shi2025} developed a stochastic approximation algorithm with favorable sample complexity for problems involving expectation-valued objective functions and both equality and inequality constraints. Therefore, improving the complexity of PMQSopt remains an important direction from a complexity-theoretic perspective.

For general nonconvex stochastic optimization problems under weaker conditions than those considered in this paper, analyzing the sample complexities remains a challenging but interesting direction for future research.

\end{document}